\documentclass[12pt]{article}

\usepackage{amsmath,amssymb,amscd,amsthm,esint}
\usepackage{graphics,amsmath,amssymb,amsthm,mathrsfs}
\usepackage{amssymb,amsfonts}
\usepackage[all,arc]{xy}
\usepackage{enumerate}
\usepackage{mathrsfs}
\usepackage{amsmath,cite,amsthm}

\newtheorem{theorem}{Theorem}[section]

\newtheorem{lemma}[theorem]{Lemma}

\newtheorem{corollary}[theorem]{Corollary}

\newtheorem{remark}[theorem]{Remark}

\newcommand{\average}{-\!\!\!\!\!\!\int}
\newcommand{\varep}{\varepsilon}

\begin{document}

\title
{\bf Homogenization of Stokes Systems \\ and Uniform Regularity Estimates}

\author{Shu Gu\thanks{Supported in part by NSF grant DMS-1161154.}
\and Zhongwei Shen\thanks{Supported in part by NSF grant DMS-1161154.}}

\date{ }

\maketitle

\begin{abstract}

This paper is concerned  with uniform regularity estimates for a family of Stokes systems with rapidly oscillating periodic coefficients.
We establish  interior Lipschitz estimates for the velocity and $L^\infty$ estimates for the pressure
as well as a Liouville property for solutions in $\mathbb{R}^d$.
We also obtain the boundary $W^{1,p}$ estimates in a bounded $C^1$ domain 
for any $1<p<\infty$.

\medskip

\noindent{\it Keywords.} Homogenization; Stokes systems; Regularity.

\medskip

\noindent{\it AMS Subject Classifications.} 35B27; 35J48.

\end{abstract}



\section{Introduction and Main Results}
\setcounter{equation}{0}

The primary purpose of this paper is to establish uniform regularity estimates
in the homogenization theory of Stokes systems with rapidly oscillating periodic coefficients.
More precisely, we consider the Stokes system in fluid dynamics,
\begin{equation}\label{Stokes}
\left\{
\aligned
 \mathcal{L}_\varep (u_\varep) +\nabla p_\varep & = F,\\
 \text{ div}( u_\varep ) & = g
\endaligned
\right.
\end{equation}
in a bounded domain $\Omega$ in $\mathbb{R}^d$,
where $\varep>0$ and
\begin{equation}\label{Second-order Elliptic Operator}
\mathcal{L}_\varep = -\text{div}\big(A(x/\varepsilon)\nabla\big)=
- \frac{\partial}{\partial x_i}\Bigl[a_{ij}^{\alpha\beta}\Bigl(\frac{x}{\varepsilon}\Bigr)\frac{\partial}{\partial x_j}\Bigr ]
\end{equation}
with $1\le i, j, \alpha, \beta\le d$ (the summation convention is used throughout).
We will assume that the coefficient matrix $A(y)=\big(a_{ij}^{\alpha\beta}(y)\big)$ is real, bounded measurable, 
and satisfies the ellipticity condition:
\begin{equation}\label{ellipticity}
\mu |\xi |^2 \le a_{ij}^{\alpha\beta} (y)\xi_i^\alpha\xi_j^\beta \le \frac{1}{\mu}|\xi|^2  \quad
\text{ for } y\in\mathbb{R}^d \text{ and } \xi=(\xi_i^\alpha)\in \mathbb{R}^{d\times d},
\end{equation}
where $\mu>0$, and the periodicity condition:
\begin{equation}\label{periodicity}
A(y+z)=A(y) \quad \text{ for } y\in \mathbb{R}^d \text{ and } z \in \mathbb{Z}^d.
\end{equation}
A function satisfying (\ref{periodicity}) will be called 1-periodic.
We note that
the system (\ref{Stokes}), which does not fit the standard framework of second-order elliptic systems considered in 
\cite{AL-1987, KLS1},
 is used in the modeling of flows in porous media.
 
The following is one of the main results of the paper.

\begin{theorem}\label{main-theorem-1}
Suppose that $A(y)$ satisfies the ellipticity condition (\ref{ellipticity})
and periodicity condition (\ref{periodicity}).
Let $(u_\varep, p_\varep)$ be a weak solution of the Stokes system (\ref{Stokes}) in 
$B(x_0, R)$ for some $x_0\in \mathbb{R}^d$ and $R>\varep$.
Then, for any $\varep\le r< R$,
\begin{equation}\label{main-estimate-1}
\aligned
&\left(\average_{B(x_0, r)} |\nabla u_\varep|^2\right)^{1/2}
+\left(\average_{B(x_0, r)} \big|p_\varep -\average_{B(x_0, R)}p_\varep\big|^2\right)^{1/2}\\
&\le C \left\{  \left(\average_{B(x_0, R)} |\nabla u_\varep|^2\right)^{1/2}
+\| g \|_{L^\infty(B(x_0, R))} 
+R^\rho [ g]_{C^{0, \rho}(B(x_0, R))}\right\}\\
& \qquad\qquad\qquad\qquad
+C\, R \left( \average_{B(x_0, R)} |F|^q\right)^{1/q},
\endaligned
\end{equation}
where $0<\rho=1-\frac{d}{q}<1$, and the constant $C$ depends only on $d$, $\mu$, and $\rho$.
\end{theorem}

The scaling-invariant estimate (\ref{main-estimate-1}) should be regarded as a Lipschitz estimate for the velocity $u_\varep$
 and $L^\infty$ estimate for the pressure $p_\varep$ down to the microscopic 
 scale $\varep$, even though no smoothness assumption is made on the coefficient matrix $A(y)$. 
 Indeed, if estimate (\ref{main-estimate-1}) holds for any $0<r<R$, we would be able to bound 
 $$
 |\nabla u_\varep (x_0)| + | p_\varep (x_0) -\average_{B(x_0, R)} p_\varep|
 $$
 by the right hand side of (\ref{main-estimate-1}).
 Here we have taken a point of view  that 
 solutions should behave much better on mesoscopic scales due to homogenization 
 and that the smoothness of coefficients only effects the solutions below the microscopic scale
 (see this viewpoint in the recent development on quantitative stochastic homogenization 
 in \cite{Armstrong-Smart-2014, GNO-2014} and their references).
 In fact, under the additional assumption that $A(y)$ is H\"older continuous,
 \begin{equation}\label{holder}
 |A(x)-A(y)|\le \tau |x-y|^\lambda \qquad \text{ for } x,y\in \mathbb{R}^d,
 \end{equation}
 where $\lambda\in (0,1]$ and $\tau>0$,
 we may deduce the full uniform Lipschitz estimate for $u_\varep$ and $L^\infty$ estimate
 for $p_\varep$ from Theorem \ref{main-theorem-1}, by a  blow-up argument (see Section 5). 
 
 \begin{corollary}\label{corollary-1}
Suppose that $A(y)$ satisfies conditions (\ref{ellipticity}), 
(\ref{periodicity}) and (\ref{holder}).
Let $(u_\varep, p_\varep)$ be a weak solution of (\ref{Stokes})
in $B(x_0, R)$ for some $x_0\in \mathbb{R}^d$ and $R>0$.
Then
\begin{equation}\label{Lip-estimate}
\aligned
&\| \nabla u_\varep\|_{L^\infty(B(x_0, R/2))}
+\| p_\varep -\average_{B(x_0, R)} p_\varep \|_{L^\infty(B(x_0, R/2))} \\
&\le C \left\{  \left(\average_{B(x_0, R)} |\nabla u_\varep|^2\right)^{1/2}
+\| g\|_{L^\infty(B(x_0, R))} 
+R^\rho [ g]_{C^{0, \rho}(B(x_0, R))}\right\}\\
& \qquad\qquad\qquad\qquad
+C\, R \left( \average_{B(x_0, R)} |F|^q\right)^{1/q},
\endaligned
\end{equation}
where $0<\rho=1-\frac{d}{q}$, and the constant $C$ depends only on $d$, $\mu$, $\lambda$, $\tau$,
and $\rho$.
\end{corollary}
 
We remark that for the standard second-order elliptic system $\mathcal{L}_\varep (u_\varep)=F$,
 uniform interior Lipschitz estimates as well as uniform boundary Lipchitz estimates with Dirichlet
 conditions in $C^{1, \alpha}$ domains, were established by M. Avellaneda and F. Lin in \cite{AL-1987},
 under conditions (\ref{ellipticity}), (\ref{periodicity}) and (\ref{holder}).
 Under the additional symmetry condition $A^*=A$, the boundary
 Lipschitz estimates with Neumann boundary conditions in $C^{1,\alpha}$ domains
 were obtained by C. Kenig, F. Lin, and Z. Shen in \cite{KLS1}.
 This symmetry condition was recently removed by S.N.  Armstrong and Z. Shen in \cite{Armstrong-Shen-2014},
 where the uniform Lipschitz estimates were studied for second-order elliptic systems
 in divergence form with almost-periodic coefficients.
 
  The proof of Theorem \ref{main-theorem-1}, given in Sections 3 and 5,
   uses a compactness argument, which was introduced to the
 study of homogenization problems by M. Avellaneda and F. Lin \cite{AL-1987,AL-1989-II}.
 Let $(u_\varep, p_\varep)$ be a weak solution of the Stokes system (\ref{Stokes}) in
 $B(0,1)$. Suppose that
 $$
 \max\left\{ \left(\average_{B(0,1)} |u_\varep|^2\right)^{1/2},
 \left(\average_{B(0,1)} |F|^q\right)^{1/q},
 \| g\|_{C^\rho(B(0,1))} \right\}\le 1,
 $$ 
where $\rho=1-\frac{d}{q}>0$.
By the compactness argument with an iteration procedure, 
which is more or less the $L^2$ version of the compactness method used in \cite{AL-1987},
we are able to show that if $0<\varep<\theta^{\ell-1}\varep_0$
for some $\ell\ge 1$, then
\begin{equation}\label{1-10}
\left(\average_{B(0, \theta^\ell)}
\big|u_\varep -\big(P_j^\beta (x) +\varep\chi_j^\beta (x/\varep) \big) E_j^\beta (\varep, \ell) -G(\varep, \ell) \big|^2\,dx \right)^{1/2}
\le \theta^{\ell (1+\sigma)},
\end{equation}
where $0<\sigma<\rho$, and $E_j^\beta (\varep, \ell)$, $G(\varep, \ell)$
are constants satisfying $|E_j^\ell(\varep, \ell)| +|G(\varep, \ell)|\le C$ (see Lemma \ref{step-2}).
In (\ref{1-10}), $P_j^\beta (y)=y_j (0, \dots, 1, \dots)$ with $1$ in the $\beta^{th}$ position
and $\chi=(\chi_j^\beta (y))$ is the so-called corrector  associated with the Stokes system (\ref{Stokes}).
We remark that estimate (1.8) may be regarded as a $C^{1, \sigma}$ estimate for $u_\varep$
in scales larger than $\varep$.
This estimate allows us to deduce the Lipschitz estimate for the velocity $u_\varep$ down to the scale $\varep$
(see Section 3).
Moreover, by carefully analyzing the error terms in the asymptotic expansion of $p_\varep$,
the estimate (\ref{1-10}) also allows us to bound
$$
\Big| \average_{B(x_0,r)} p_\varep -\average_{B(x_0, R)} p_\varep\Big|
$$
and to derive the $L^\infty$ estimate for  the pressure $p_\varep$,
 one of the main novelties of this paper (see Section 5).
 We remark that the control of pressure terms usually requires new ideas in the study of Stokes or
 Navier-Stokes systems.
In our case $p_\varep$ is related to $\nabla u_\varep$ by singular integrals
that  are not  bounded on $L^\infty$;
Lipschitz estimates for $u_\varep$ in general do not imply $L^\infty$ estimates for $p_\varep$.
Also, observe that our $L^2$ formulation in (\ref{1-10}), in comparison with the $L^\infty$
setting used in \cite{AL-1987,KLS1}, appears to be necessary,
as the correctors are not necessarily bounded without smoothness conditions on $A$.
We further note that as a consequence of (\ref{1-10}),
we are able to establish a Liouville property for Stokes systems with periodic coefficients (see Section 4).
To the best of authors' knowledge, this appears to be the first result on the Liouville property
for Stokes systems with variable coefficients.

In this paper we also study the uniform boundary regularity estimates for (\ref{Stokes}) in $C^1$ domains.
The following theorem, whose proof is given in Section 6,
 may be regarded as a boundary H\"older estimate for $u_\varep$ down to the scale $\varep$.
 We emphases that as in the case of Theorem \ref{main-theorem-1},
  no smoothness assumption on $A$ is required for Theorem \ref{main-theorem-2}.

 \begin{theorem}\label{main-theorem-2}
 Suppose that $A(y)$ satisfies conditions (\ref{ellipticity}) and (\ref{periodicity}).
 Let $\Omega$ be a bounded $C^1$ domain in $\mathbb{R}^d$.
 Let $x_0\in \partial\Omega$ and $0<R<R_0$, where $R_0=\text{\rm diam}(\Omega)$.
 Let $(u_\varep, p_\varep)$ be a weak solution of
 \begin{equation}\label{boundary-equation}
 \left\{
 \aligned
 \mathcal{L}_\varep (u_\varep) +\nabla p_\varep  & =0 &\quad & \text{ in } B(x_0, R)\cap \Omega,\\
 \text{\rm div} (u_\varep) & =0&\quad & \text{ in } B(x_0, R)\cap \Omega,\\
 u_\varep  & =0 & \quad &\text{ on } B(x_0, R)\cap \partial\Omega.
 \endaligned
 \right.
 \end{equation}
 Suppose that $0<\varep\le r< R$ and $0<\rho<1$.
 Then
 \begin{equation}\label{boundary-holder}
 \left(\average_{B(x_0, r)\cap\Omega} |\nabla u_\varep|^2\right)^{1/2}
 \le C_\rho \left(\frac{r}{R} \right)^{\rho-1}
 \left(\average_{B(x_0, R)\cap\Omega} |\nabla u_\varep|^2\right)^{1/2},
 \end{equation}
 where $C_\rho$ depends only on $d$, $\mu$, $\rho$, and $\Omega$.
 \end{theorem}
 
 Theorem \ref{main-theorem-2} is also proved by a compactness method, 
 though correctors are not needed here. The scaling-invariant boundary estimate (\ref{boundary-holder}),
  combined with the interior estimates
 in Theorem \ref{main-theorem-1}, allows us to establish the boundary $W^{1,p}$
 estimates for Stokes systems with $V\!M\!O$ coefficients  in $C^1$ domains.
 
 Let $B^{\alpha, q}(\partial \Omega; \mathbb{R}^d)$ denote the Besov space of $\mathbb{R}^d$-valued functions on
 $\partial\Omega$ of order $\alpha \in (0,1) $ with exponent $q\in (1, \infty)$.
 It is known that if $u\in W^{1, q}(\Omega; \mathbb{R}^d)$ for some $1<q<\infty$,
 where $\Omega$ is a bounded Lipschitz domain, then $u|_{\partial\Omega}
 \in B^{1-\frac{1}{q}, q}(\partial \Omega; \mathbb{R}^d)$.
 
 \begin{theorem}\label{main-theorem-3}
Let $\Omega$ be a bounded $C^1$ domain in $\mathbb{R}^d$ and $1<q<\infty$. 
Suppose that $A$ satisfies conditions (\ref{ellipticity}) and (\ref{periodicity}).
Also assume that $A\in V\!M\!O (\mathbb{R}^d)$.
Let $f=(f_i^\alpha)\in L^q(\Omega;\mathbb{R}^{d\times d})$, 
$g \in L^q(\Omega)$ and $h \in B^{1-\frac{1}{q},q}(\partial\Omega;\mathbb{R}^d)$ satisfy the 
compatibility condition
\begin{equation}\label{comp}
\int_\Omega g -\int_{\partial\Omega} h \cdot n=0,
\end{equation}
where $n$ denotes the outward unit normal to $\partial\Omega$.
Then the solutions $(u_\varepsilon, p_\varep)$ in $W^{1,q}(\Omega;\mathbb{R}^d)$ to Dirichlet problem
\begin{equation}\label{DP}
\left\{
\aligned
 \mathcal{L}_\varepsilon(u_\varepsilon) +\nabla p_\varep
 &= \text{\rm div}(f) &\quad & \text{ in }\Omega,\\
 \text{\rm div}(u_\varepsilon) &  =g &\quad &  \text{ in }\Omega,\\
 u_\varepsilon &= h & \quad & \text{ on }\partial\Omega,
\endaligned
\right.
\end{equation}
satisfy the estimate
\begin{equation}\label{W-1-p}
 \|\nabla u_\varepsilon\|_{L^q(\Omega)} 
+\| p_\varep -\average_\Omega p_\varep \|_{L^q(\Omega)}\\
 \le C_q \left\{\|f\|_{L^q(\Omega)}+\|g_\varepsilon\|_{L^q(\Omega)}
+\|h\|_{B^{1-\frac{1}{q},q}(\partial\Omega)}\right\},
\end{equation}
where $C_q$ depends only on $d$, $q$, $A$ and $\Omega$.
\end{theorem}

The proof of Theorem \ref{main-theorem-3} is given in Sections 7 and 8.
We mention that $W^{1, p}$ estimates for elliptic and parabolic equations with continuous or
$V\!M\!O$ coefficients have been studied extensively in recent years. We refer the reader to \cite {
CP-1998, Byun-Wang-2004, Shen-2005, Krylov-2007, Kim-Krylov-2007,
Byun-Wang-2008, Geng-2012} as well as their references for work on elliptic equations and systems, 
and  to \cite { AL-1987, AL-1991, CP-1998, Shen-2008, KLS1, song-2012, Geng-Shen-2014} 
for uniform $W^{1,p}$ estimates in homogenization.

We end this section with some notations and observations.
We will use $\average_E f =\frac{1}{|E|} \int_E f $ to denote the $L^1$ average of $f$ over the set $E$.
We will use $C$ to denote constants that may depend on $d$, $A$, or $\Omega$, but never on $\varep$.
Note that our assumptions on $A$ are invariant under translation.
Finally, the technique of rescaling (or dilation) will be used routinely in the rest of the paper.
For this, we record that if $(u_\varep, p_\varep)$ is a solution of (\ref{Stokes}) and
$v(x)=u_\varep (rx)$, then
\begin{equation}\label{rescaling}
\left\{
\aligned
\mathcal{L}_{\frac{\varep}{r}} (v)+\nabla \pi  & =G,\\
\text{\rm div}(v) &=h,
\endaligned
\right.
\end{equation}
where
\begin{equation}\label{rescaling-1}
\pi (x)=r p_\varep (rx), \quad h(x)= r g(rx), \quad \text{ and } \quad G(x)=r^2 F(rx).
\end{equation}
 
 \noindent{\bf Acknowledgement.}
 Both authors would like to thank the anonymous referees for their very helpful comments and suggestions.



\section{Homogenization Theorems and Compactness}
\setcounter{equation}{0}

In this section we give a review of homogenization theory of
the Stokes systems with periodic coefficients. We refer the reader to \cite[pp.76-81]{bensoussan-1978}
for a detailed presentation. We also prove a compactness theorem for a sequence of 
Stokes systems with (periodic) coefficient matrices satisfying the ellipticity condition (\ref{ellipticity})
with the same $\mu$.

Let $\Omega$ be a bounded Lipschitz domain in $\mathbb{R}^d$.
For $u, v \in H^1(\Omega;\mathbb{R}^d)$, we set
\begin{equation}\label{a}
a_\varepsilon(u,v)=\int_\Omega a_{ij}^{\alpha\beta}\Bigl(\frac{x}
{\varepsilon}\Bigr)\frac{\partial u^\beta}{\partial x_j}\frac{\partial v^\alpha}{\partial x_i}\, dx.
\end{equation}
For $F\in H^{-1}(\Omega;\mathbb{R}^d)$ and $g\in L^2(\Omega)$,
 we say that $(u_\varepsilon,p_\varepsilon)\in H^1(\Omega;\mathbb{R}^d)\times L^2(\Omega)$ 
 is a weak solution of the Stokes system (\ref{Stokes}) in $\Omega$, if div$(u_\varep)=g$ in $\Omega$ and 
 for any $\varphi \in C_0^1(\Omega;\mathbb{R}^d)$,
\begin{equation*}
a_\varepsilon(u_\varepsilon, \varphi)
-\int_\Omega p_\varep\,  \text{\rm div} (\varphi) = \langle F, \varphi \rangle.
\end{equation*}

\begin{theorem}\label{theorem-2.1}
Let $\Omega$ be a bounded Lipschitz domain in $\mathbb{R}^d$. 
Suppose  $A$ satisfies the ellipticity condition (\ref{ellipticity}). 
Let $F\in H^{-1}(\Omega;\mathbb{R}^d)$,  $g\in L^2(\Omega)$ and $h\in H^{1/2}(\partial\Omega;
\mathbb{R}^d)$ satisfy the compatibility condition (\ref{comp}).
Then there exist a unique $u_\varep \in H^1(\Omega; \mathbb{R}^d)$
and $p_\varep \in L^2(\Omega)$ (unique up to constants) such that 
$(u_\varep, p_\varep)$ is a weak solution of
(\ref{Stokes}) in $\Omega$ and $u_\varep =h$ on $\partial\Omega$.
Moreover, 
\begin{equation}\label{weak-solution-estimate}
\| u_\varepsilon\|_{H^1(\Omega)} +\| p_\varep -\average_\Omega p_\varep\|_{L^2(\Omega)}
\le C \Big\{ \| F \|_{H^{-1}(\Omega)}+ 
\| h \|_{H^{1/2}(\partial\Omega)} + \| g \|_{L^2(\Omega)}\Big\},
\end{equation}
where $C$ depends only on $d$, $\mu$, and $\Omega$.
\end{theorem}

\begin{proof}
This theorem is well known and does not use the periodicity condition of $A$.
First, we choose $\widetilde{h}\in H^1(\Omega; \mathbb{R}^d)$ such that
$\widetilde{h}=h $ on $ \partial\Omega$
  and $ \| \widetilde{h}\|_{H^1(\Omega)}\le C\, \| h\|_{H^{1/2} (\partial\Omega)}$.
By considering $u_\varep -\widetilde{h}$,
we may assume that $h=0$.
Next, we choose a function $U(x)$ in $ H^1_0 (\Omega; \mathbb{R}^d)$
such that div$(U)=g$ in $\Omega$ and $\|U\|_{H^1(\Omega)} \le C\, \| g\|_{L^2(\Omega)}$
(see \cite{Duran-2012} for a proof of the existence of such functions).
By considering $u_\varep -U$, we may further assume that $g=0$.
Finally, the case $h=0$ and $g=0$ may be proved
by applying the Lax-Milgram Theorem to the bilinear form $a_\varep (u,v)$
on the Hilbert space
$$
V=\big\{ u\in H^1_0(\Omega; \mathbb{R}^d): \, \text{\rm div} (u)=0  \text{ in } \Omega\big\}.
$$
This completes the proof.
\end{proof}

Let $Y=[0,1)^d$. We denote by $H^1_{\text{per}}(Y;\mathbb{R}^d)$ 
the closure in $H^1(Y;\mathbb{R}^d)$ of $C^\infty_{\text{per}}(Y;\mathbb{R}^d)$, 
the set of $C^\infty$ 1-periodic and  $\mathbb{R}^d$-valued functions in $\mathbb{R}^d$.
Let
\begin{equation*}
a_{\text{per}} (\psi, \phi )=\int_Y a_{ij}^{\alpha\beta}(y)\frac{\partial \psi^\beta}{\partial y_j}
\frac{\partial \phi^\alpha}{\partial y_i}\, dy,
\end{equation*}
where $\phi=(\phi^\alpha)$ and $\psi=(\psi^\alpha)$.
By applying the Lax-Milgram Theorem to the bilinear form $a_{\text{per}} (\psi, \phi)$ on
the Hilbert space
$$
V_{\text{per}}(Y)
=\Big\{ u\in H^1_{\text{per}} (Y; \mathbb{R}^{d}): \, \text{div} (u)=0 \text{ in } Y  \text{ and } \int_Y u=0 \Big\},
$$
it follows that  for each $1\le j, \beta\le d$,
 there exists a unique $\chi_j^\beta \in V_{\text{per}}(Y)$ such that
 $$
 a_{\text{per}} (\chi_j^\beta, \phi) =-a_{\text{per}}
 (P_j^\beta, \phi) \qquad \text{ for any } \phi \in V_{\text{per}}(Y),
 $$
 where $P_j^\beta=P_j^\beta(y)=y_je^\beta=y_j(0,...,1,...,0)$ with 1 in the $\beta^{th}$ position.
 As a result, there exist 1-periodic functions $(\chi_j^\beta, \pi_j^\beta)
 \in H^1_{\text{loc}}(\mathbb{R}^d; \mathbb{R}^d) \times L^2_{\text{loc}}(\mathbb{R}^d)$,
 which are called the correctors for the Stokes system (\ref{Stokes}),
 such that
\begin{equation}\label{corrector}
\left\{
\aligned
\mathcal{L}_1(\chi_j^\beta +P_j^\beta) +\nabla \pi_j^\beta  &= 0 \ \text{ in } \mathbb{R}^d,\\
 \text{ div} (\chi_j^\beta) & =0 \ \text{ in }\mathbb{R}^d,\\
 \int_Y \pi_j^\beta=0 \text{ and }  \int_Y \chi_j^\beta & =0.
\endaligned
\right.
\end{equation}
Note that
\begin{equation}\label{corrector-estimate}
\| \chi_j^\beta \|_{H^1(Y)} +\|\pi_j^\beta\|_{L^2(Y)} \le C,
\end{equation}
where $C$ depends only on $d$ and $\mu$.
Let $\widehat{A}=\big(\widehat{a}_{ij}^{\alpha\beta}\big)$, where
\begin{equation}\label{homo-coefficient}
\widehat a_{ij}^{\alpha\beta}=a_{\text{per}} \big(\chi_j^\beta+P_j^\beta, \chi_i^\alpha+P_i^\alpha\big).
\end{equation}
The homogenized system for the Stokes system (\ref{Stokes}) is given by
\begin{equation}\label{homo-system}
\left\{
\aligned
\mathcal{L}_0 (u_0) +\nabla p_0  & = F,\\
\text{\rm div} (u_0) & =g,
\endaligned
\right.
\end{equation}
where $\mathcal{L}_0 =-\text{\rm div} (\widehat{A}\nabla )$ is
a second-order elliptic operator with constant coefficients.

\begin{remark}
{\rm
The homogenized matrix $\widehat{A}$ satisfies the ellipticity condition
\begin{equation}\label{ellipticity-1}
\mu |\xi|^2 \le \widehat{a}_{ij}^{\alpha\beta} \xi_i^\alpha \xi_j^\beta
\le \mu_1 |\xi|^2
\end{equation}
for any $\xi =(\xi_i^\alpha) \in \mathbb{R}^{d\times d}$, where $\mu_1$ depends only on $d$ and $\mu$.
The upper bound is a consequence of the estimate
$\|\nabla \chi_j^\beta\|_{L^2(Y)} \le C(d, \mu)$,
while the lower bound follows from
$$
\aligned
\widehat{a}_{ij}^{\alpha\beta}\xi_i^\alpha \xi_j^\beta
&=a_{\text{per}} \big( (\chi_j^\beta +P_j^\beta)\xi_j^\beta, (\chi_i^\alpha +P_i^\alpha)\xi_i^\alpha\big)\\
&\ge \mu \int_Y |\nabla (\chi_i^\alpha +P_i^\alpha)\xi_i^\alpha|^2\\
&\ge \mu| \xi|^2.
\endaligned
$$
}
\end{remark}

\begin{remark}\label{adjoint-remark}
{\rm
Let $\chi^* =(\chi_j^{*\beta})$ denote the matrix of correctors for the system (\ref{Stokes}), with $A$ replaced
by its adjoint $A^*$.
Note that by definition,
$\chi_j^{*\beta}\in V_{\text{per}}(Y)$ and
$$
a_{\text{per}}^*(\chi_j^{*\beta}, \phi)= -a_{\text{per}}^*(P_j^\beta, \phi) \qquad \text{ for any } \phi \in V_{\text{per}}(Y),
$$
where $a_{\text{per}} ^*(\psi, \phi)=a_{\text{per}} (\phi, \psi)$.
It follows that \begin{equation}\label{adjoint-form}
\aligned
\widehat{a}_{ij}^{\alpha\beta}
& =a_{\text{per}}\big(\chi_j^\beta +P_j^\beta, \chi_i^\alpha +P_i^\alpha\big)
=a_{\text{per}}\big(\chi_j^\beta +P_j^\beta, P_i^\alpha\big)\\
&=a_{\text{per}}\big(\chi_j^\beta +P_j^\alpha, \chi_i^{*\alpha} +P_i^\alpha\big)
=a^*_{\text{per}} \big(\chi_i^{*\alpha} +P_i^\alpha, \chi_j^{\beta} +P_j^\beta\big)\\
&=a^*_{\text{per}} \big(\chi_i^{*\alpha} +P_i^\alpha, P_j^\beta\big)
=a^*_{\text{per}} \big(\chi_i^{*\alpha} +P_i^\alpha, \chi_j^{*\beta} +P_j^\beta\big).
\endaligned
\end{equation}
This, in particular, shows that $\big( \widehat{A} \big)^* =\widehat{A^*}$.
}
\end{remark}

\begin{theorem}\label{homogenization-theorem}
Suppose that $A(y)$ satisfies conditions (\ref{ellipticity}) and (\ref{periodicity}).
 Let $\Omega$ be a bounded Lipschitz domain.
  Let $(u_\varepsilon, p_\varep) \in H^1(\Omega;\mathbb{R}^d)\times L^2(\Omega)$ 
  be a  weak solution of
$$
\left \{
\aligned
 \mathcal{L}_\varepsilon(u_\varepsilon) +\nabla p_\varep & = F &\quad &\text{ in } \Omega, \\
 \text{\rm div} (u_\varepsilon) & = g &\quad &\text{ in } \Omega,\\
 u_\varepsilon  &= h &\quad & \text{ on } \partial \Omega,
\endaligned
\right .
$$
where $F\in H^{-1}(\Omega; \mathbb{R}^d)$, $g\in L^2(\Omega)$ and $h\in H^{1/2}(\partial\Omega; \mathbb{R}^d)$.
Assume that $\int_\Omega p_\varep =0$.
Then as $\varepsilon\rightarrow 0$,
\begin{equation*}
\left \{
\aligned
u_\varepsilon & \rightarrow u_0 \text{ strongly in }L^2(\Omega;\mathbb{R}^d),\\
u_\varepsilon & \rightharpoonup u_0 \text{ weakly in }H^1(\Omega;\mathbb{R}^d),\\
p_\varepsilon & \rightharpoonup p_0 \text{ weakly in }L^2(\Omega),\\
A(x/\varep)\nabla u_\varep & \rightharpoonup \widehat{A}\nabla u_0 \text{ weakly in } L^2(\Omega; \mathbb{R}^{d\times d}).
\endaligned
\right.
\end{equation*}
Moreover,  $(u_0, p_0)$ is the weak solution of the homogenized problem
$$
\left \{
\aligned
 \mathcal{L}_0(u_0) +\nabla p_0  &= F &\quad&\text{ in } \Omega, \\
 \text{\rm div} (u_0)  & = g &\quad & \text{ in } \Omega,\\
 u_0  & = h &\quad & \text{ on } \partial \Omega.
\endaligned
\right.
$$
\end{theorem}

We remark that Theorem \ref{homogenization-theorem} is more or less proved in \cite {bensoussan-1978},
using Tartar's testing function method.
Our next theorem extends Theorem \ref{homogenization-theorem} to
a sequence of systems with coefficient matrices satisfying the same conditions
and should be regarded as a compactness property of the Stokes systems with periodic coefficients.
Its proof follows the same line of argument found in \cite{bensoussan-1978} for the proof of Theorem \ref{homogenization-theorem}, and also uses the observation that
if $\{ w_k\}$ is a sequence of 1-periodic functions with $\| w_k\|_{L^2(Y)} \le C$ and $\varep_k \to 0$,
then
\begin{equation}\label{periodic-convergence}
w_k (x/\varep_k) -\average_Y  w_k \rightharpoonup 0 \text{ weakly in } L^2(\Omega)
\end{equation}
as $k\to \infty$.

\begin{theorem}\label{compactness-theorem}
Let $\{ A^k(y)\} $ be a sequence of 1-periodic matrices satisfying the ellipticity condition (\ref{ellipticity}) (with the same $\mu$).
 Let $\Omega$ be a bounded Lipschitz domain
in $\mathbb{R}^d$. Let $(u_k, p_k) \in H^1(\Omega;\mathbb{R}^d)\times L^2(\Omega)$ be a weak solution of
$$
\left \{
\aligned
-\text{\rm div} \big(A^k (x/\varep_k)\nabla u_k ) +\nabla p_k  & = F_k, \\
 \text{\rm div} (u_k)  & = g_k
\endaligned
\right .
$$
in $\Omega$, where $\varepsilon_k \rightarrow 0$, $F_k \in H^{-1} (\Omega; \mathbb{R}^d)$
and $g_k \in L^2(\Omega)$.
We further assume that as $k \rightarrow \infty$,
\begin{equation*}
\left \{
\aligned
F_k & \rightarrow F_0 \text{ strongly in } H^{-1} (\Omega; \mathbb{R}^d),\\
g_k & \rightarrow g_0 \text{ strongly in }L^2(\Omega),\\
u_k & \rightharpoonup u_0 \text{ weakly in } H^1(\Omega;\mathbb{R}^d),\\
p_k & \rightharpoonup p_0 \text{ weakly in }L^2(\Omega),\\
\widehat{A^k}  & \rightarrow A^0,
\endaligned
\right.
\end{equation*}
where $\widehat {A^k}$ is the coefficient matrix of the homogenized system
for the Stokes system with coefficient matrix $A^k(x/\varep)$.
Then, $A^k(x/{\varepsilon_k})\nabla u_k \rightharpoonup  A^0\nabla u_0$ weakly in $L^2(\Omega;\mathbb{R}^{d\times d})$,
and $(u_0, p_0)$ is a weak solution of 
\begin{equation}\label{2.2-1}
\left \{
\aligned
-\text{\rm div} \big (A^0\nabla u_0 \big) +\nabla p_0  &= F_0, \\
 \text{\rm div} (u_0)   & = g_0
\endaligned
\right.
\qquad\qquad
\text{ in } \  \Omega.
\end{equation}
\end{theorem}

\begin{proof}
Let $A^k = \big( a_{ij}^{k\alpha\beta}\big) $ and
\begin{equation*}
(\xi_k)_i^\alpha = {a}_{ij}^{k\alpha\beta}\left(\frac{x}{\varepsilon_k}\right)\frac{\partial u_k^\beta}{\partial x_j}.
\end{equation*}
Note that
$
\| (\xi_k)_i^\alpha \|_{L^2(\Omega)} \le \mathit{C}
$.
It suffices to show that if $\{ \xi_{k^\prime}\}$ is a subsequence of $\{ \xi_k\}$ and $\xi_{k^\prime}$
converges weakly to $\xi_0$ in $L^2(\Omega; \mathbb{R}^{d\times d})$, then $\xi_0 =A^0\nabla u_0$.
This would imply that $(u_0, p_0)$ is a weak solution of (\ref{2.2-1}) in $\Omega$.
It also implies  that the whole sequence $\xi_k$ converges weakly to $A^0\nabla u_0$ in
$L^2(\Omega; \mathbb{R}^{d\times d})$.

Without loss of generality we may assume that $\xi_k\rightharpoonup  \xi_0$ weakly 
in $L^2(\Omega; \mathbb{R}^{d\times d})$.
Note that
\begin{equation}\label{2.2-5}
\langle \xi_k, \nabla \phi \rangle = \langle F_k, \phi \rangle+\langle p_k,\text{div}( \phi) \rangle
\end{equation}
for all $\phi \in H_0^1(\Omega;\mathbb{R}^d)$.
Fix $1\le j, \beta\le d$ and $\psi \in C_0^1 (\Omega)$.
Let
$$
\phi_k (x)= \left( P_j^\beta (x) +\varep _k \chi_j^{k*\beta} (x/\varep_k) \right) \psi (x),
$$
where $\chi^{k*\beta}_j$  (and $\pi^{k*\beta}_j$ used in the following)
are the correctors for the Stokes system with coefficient matrix
$(A^k)^*(x/\varep)$, introduced in Remark \ref{adjoint-remark}.
A computation shows that
\begin{equation}\label{2.2-10}
\aligned
\langle \xi_k, \nabla \phi_k \rangle
&=\langle A^k (x/\varep_k)\nabla u_k, \nabla \big(P_j^\beta +\varep_k \chi_j^{k*\beta} (x/\varep_k) \big)\cdot \psi\rangle\\
& \qquad\qquad\qquad
+\langle A^k (x/\varep_k)\nabla u_k, \big(P_j^\beta +\varep_k \chi_j^{k*\beta} (x/\varep_k)\big) \nabla \psi \rangle\\
&=\langle \psi (\nabla u_k), (A^k)^* (x/\varep_k) \nabla \big(P_j^\beta +\varep_k \chi_j^{*\beta} (x/\varep_k) \big)\rangle\\
& \qquad\qquad\qquad
+\langle A^k (x/\varep_k)\nabla u_k, \big(P_j^\beta +\varep_k \chi_j^{k*\beta} (x/\varep_k) \big)\nabla \psi \rangle\\
&=\langle \nabla (\psi u_k), (A^k)^* (x/\varep_k) \nabla \big(P_j^\beta +\varep_k \chi_j^{k*\beta} (x/\varep_k) \big)\rangle\\
&\qquad\qquad\qquad
-\langle (\nabla \psi) u_k, (A^k)^* (x/\varep_k) \nabla \big(P_j^\beta +\varep_k \chi_j^{k*\beta} (x/\varep_k) \big)\rangle\\
& \qquad\qquad\qquad
+\langle \xi_k, \big(P_j^\beta +\varep_k \chi_j^{k*\beta} (x/\varep_k) \big)\nabla \psi \rangle.\\
\endaligned
\end{equation}
Since
$$
-\text{\rm div} \left( (A^{k})^*(x/\varep_k) \nabla \left[ P_j^\beta +\varep_k \chi_j^{k*\beta} (x/\varep_k) \right]\right)
=-\nabla \left[  \pi_j^{k*\beta} (x/\varep_k) \right] \qquad \text{ in } \mathbb{R}^d,
$$
it follows that the first term in the right hand side of (\ref{2.2-10}) equals 
$$
\langle  \pi_j^{k*\beta} (x/\varep_k), \text{\rm div} (\psi u_k) \rangle
=\langle  \pi_j^{k*\beta} (x/\varep_k) -\average_Y \pi_j^{k*\beta}, \text{\rm div} (\psi u_k) \rangle.
$$
Using the fact that
$\text{div}(\psi u_k)=\nabla \psi\cdot u_k +\psi g_k  \to \nabla \psi \cdot u_0 +\psi g_0$ strongly in $L^2(\Omega)$
and
$$
\pi_j^{k*\beta} (x/\varep_k) -\average_Y \pi_j^{k*\beta} \rightharpoonup 0 \text{ weakly in } L^2(\Omega),
$$
we see that the first term in the right hand side of (\ref{2.2-10}) goes to zero.
In view of the estimate 
$$
\|\varep_k \chi_j^{k*\beta} (x/\varep_k)\|_{L^2(\Omega)}
\le C\, \varep_k \| \chi_j^{k*\beta} \|_{L^2(Y)} \le C\, \varep_k,
$$
it is easy to see that the third term in the right side of (\ref{2.2-10}) goes to 
$\langle \xi_0, P_j^\beta \nabla \psi \rangle$.

To handle the second term in the right hand side of (\ref{2.2-10}),
we note that by (\ref{periodic-convergence}),
$$
 \nabla P_i^\alpha \cdot  (A^k)^* (x/\varep_k) \nabla \big(P_j^\beta +\varep_k \chi_j^{k*\beta} (x/\varep_k) \big)
 $$
 converges weakly in $L^2(\Omega)$ 
 to
 $$
  \lim_{k\to \infty}
 \int_Y \nabla P_i ^\alpha \cdot (A^{k})^* (y) \nabla \big( P_j^\beta + \chi_j^{k*\beta} (y) \big)\, dy
=\lim_{k\to \infty} \widehat{a}_{ji}^{k\beta \alpha}
={a}^{0\beta\alpha}_{ji},
$$
where $\widehat{A^k}= (\widehat{a}_{ij}^{k\alpha\beta} )$,
$A^0= ( a^{0\alpha\beta}_{ij})$, and we have used the observation (\ref{adjoint-form}).
This, together with the fact that $u_k\to u_0$ strongly in $L^2(\Omega; \mathbb{R}^d)$, shows that
the second term in the right hand side of (\ref{2.2-10}) goes to 
$$
-{a}_{ji}^{0\beta\alpha} \int_\Omega \frac{\partial \psi}{\partial x_i} u_0^\alpha
={a}_{ji}^{0\beta\alpha} \int_\Omega \psi \frac{\partial u_0^\alpha}{\partial x_i},
$$
where we have used integration by parts.
To summarize, we have proved that as $k\to \infty$,
\begin{equation}\label{2.2-15}
\langle \xi_k, \nabla \phi_k \rangle
\to \langle \xi_0, P_j^\beta \nabla \psi\rangle 
+{a}_{ji}^{0\beta\alpha} \int_\Omega \psi \frac{\partial u_0^\alpha}{\partial x_i}.
\end{equation}

Finally, since $\phi_k \rightharpoonup  P_j^\beta \psi$ weakly in $H^1_0(\Omega; \mathbb{R}^d)$
and $F_k \to F_0$ strongly in $H^{-1}(\Omega; \mathbb{R}^d)$, we have
$\langle F_k, \phi_k\rangle \to \langle F_0, P_j^\beta \psi\rangle$.
Also, since $\text{\rm div} (\chi_j^\beta) =0$ in $\mathbb{R}^d$,
$$
\langle p_k, \text{\rm div} ( \phi_k) \rangle
=\langle p_k, \text{\rm div} (P_j^\beta \psi) \rangle+\langle p_k, \varep\chi_j^{k*\beta} (x/\varep_k) \nabla \psi\rangle
\to \langle p_0, \text{\rm div} (P_j^\beta \psi) \rangle.
$$
Thus, the right hand side of (\ref{2.2-5}) converges to
$$
\langle F_0, P_j^\beta \psi \rangle
+\langle p_0, \text{\rm div} (P_j^\beta \psi ) \rangle
=\langle \xi_0, \nabla (P_j^\beta \psi )\rangle
=\langle \xi_0, P_j^\beta \nabla \psi \rangle
+\langle \xi_0, \psi \nabla P_j^\beta\rangle,
$$
where the first equality follows by taking the limit in (\ref{2.2-5}) with $\phi=P_j^\beta \psi$.
In view of (\ref{2.2-15}) we obtain
$$
{a}_{ji}^{0\beta\alpha} \int_\Omega \psi \frac{\partial u_0^\alpha}{\partial x_i}
=\langle \xi_0, \psi \nabla P_j^\beta\rangle.
$$
Since $\psi\in C_0^1 (\Omega)$ is arbitrary, this gives
$(\xi_0)_j^\beta =a_{ji}^{0\beta\alpha} \frac{\partial u_0^\alpha}{\partial x_i}$, i.e., 
$\xi_0 =A^0 \nabla u_0$.
The proof is complete.
\end{proof}


\section{Interior Lipschitz estimates for $u_\varep$}
\setcounter{equation}{0}

For a ball
$
B=B(x_0,r)= \big\{ x\in \mathbb{R}^d: |x-x_0|<r\big\}
$
in $\mathbb{R}^d$, we will use $tB$ to denote $B(x_0,tr)$, the ball with the same center and $t$ times the radius of $B$.

We start with a Cacciopoli's inequality for the Stokes system, whose proof may be found in \cite{Gia-1982}.

\begin{theorem}
Let $(u_\varepsilon, p_\varep) \in H^1(2B; \mathbb{R}^d)\times L^2(2B)$ be a weak solution of
$$
\left\{
\aligned
 \mathcal{L}_\varepsilon(u_\varepsilon) +\nabla p_\varep
& =F+\text{\rm div} (f),\\
\text{\rm div} (u_\varepsilon) & = g
\endaligned
\right.
$$
in $2B$, where $B=B(x_0, r)$, $F\in L^2(2B; \mathbb{R}^d)$ and $f\in L^2(2B; \mathbb{R}^{d\times d})$.
Then
\begin{equation}\label{Ca}
\aligned
 \int_{B} |\nabla u_\varepsilon|^2 
& +\int_B \big|p_\varep -\average_B p_\varep\big|^2\\
& \le C \left\{
\frac{1}{r^2} \int_{2B} |u_\varepsilon|^2 +\int_{2B} |f|^2 +\int_{2B} |g|^2 + r^2 \int_{2B} |F|^2  \right\},
\endaligned
\end{equation}
where $\mathit{C}$ depends only on $d$ and $\mu$.
\end{theorem}

\begin{lemma}\label{step-1}
Let $0<\sigma<\rho<1$ and $\rho=1-\frac{d}{q}$.
Then there exist $\varep_0\in (0, 1/2)$ and $\theta\in (0,1/4)$, depending only on $d$, $\mu$, $\sigma$ and $\rho$,
such that
\begin{equation}\label{3.2-0}
\aligned
& \left(\average_{B(0, \theta)}
\big|u_\varep -\average_{B(0, \theta)} u_\varep
-(P_j^\beta +\varep \chi_j^\beta (x/\varep))
 \average_{B(0, \theta)} \frac{\partial u^\beta_\varep}{\partial x_j}\big |^2\, dx \right)^{1/2}\\
&\qquad  \le \theta^{1+\sigma}\max \left\{
 \left(\average_{B(0,1)} |u_\varep|^2\right)^{1/2},
\left(\average_{B(0,1)} |F|^q\right)^{1/q},
\| g\|_{C^\rho (B(0,1))} \right\},
\endaligned
\end{equation}
whenever $0<\varep<\varep_0$, and $(u_\varep, p_\varep)$ is a weak solution of
\begin{equation} \label{3.2-1}
\mathcal{L}_\varep  (u_\varep)+\nabla p_\varep =F \quad \text{ and } \quad \text{\rm div} (u_\varep) =g
\end{equation}
in $B(0,1)$.
\end{lemma}

\begin{proof}
We prove the lemma by contradiction, using the same approach as in \cite{AL-1987} for the elliptic system
$\mathcal{L}_\varep (u_\varep)=F$.
First, we note that by the interior $C^{1, \rho}$ estimates
for solutions of Stokes systems with constant coefficients,
\begin{equation}\label{3.2-2}
\aligned
& \left(\average_{B(0, \theta)}
\big|u_0 -\average_{B(0, \theta)} u_0
-P_j^\beta
 \average_{B(0, \theta)} \frac{\partial u_0^\beta}{\partial x_j}\big |^2 \right)^{1/2}\\
 &\qquad \le C\, \theta^{1+\rho} \| u_0\|_{C^{1, \rho} (B(0, 1/4))}\\
&\qquad  \le C_0 \, \theta^{1+\rho}\max \left\{
 \left(\average_{B(0,1/2)} |u_0|^2\right)^{1/2},
\left(\average_{B(0,1/2)} |F_0|^q\right)^{1/q},
\| g_0\|_{C^\rho (B(0,1/2))} \right\}
\endaligned
\end{equation}
for any $\theta\in (0,1/4)$, where $(u_0, p_0)$ is a weak solution of 
\begin{equation} \label{3.2-3}
-\text{\rm div} \big (A^0\nabla u_0 \big)+\nabla p_0 =F_0 \quad \text{ and } \quad \text{\rm div} (u_0) =g_0
\end{equation}
in $B(0,1/2)$ and $A^0$ is a constant matrix satisfying the ellipticity condition (\ref{ellipticity-1}).
We emphasize that the constant $C_0$ in (\ref{3.2-2}) depends only on
$d$ and $\mu$. Since $0<\sigma <\rho$, we may choose $\theta\in (1/4)$  such that
\begin{equation}\label{3.2-4-1}
2^d C_0 \theta^\rho < \theta^\sigma.
\end{equation}
We claim that there exists $\varep_0>0$, depending only on $d$, $\mu$,
$\sigma$ and $\rho$, such that the estimate (\ref{3.2-0}) holds with this $\theta$, whenever $0<\varep<\varep_0$ and
$(u_\varep, p_\varep)$ is a weak solution of (\ref{3.2-1}) in $B(0,1)$.

Suppose this is not the case. 
Then there exist sequences $\{\varep_k \}$, $\{ A^k (y)\}$, $\{ u_k \}$ and $\{ p_k\}$  such that
$\varep_k \to 0$, $A^k(y)$ satisfies (\ref{ellipticity}) and (\ref{periodicity}), 
\begin{equation}\label{3.2-4}
-\text{\rm div} \big (A^k (x/\varep_k) \nabla u_k \big) +\nabla p_k =F_k
\quad \text{ and } \quad \text{\rm div} (u_k) =g_k \quad \text{ in } B(0,1),
\end{equation}
\begin{equation}\label{3.2-5}
\max \left\{ \left(\average_{B(0,1)} |u_k|^2 \right)^{1/2},
\left(\average_{B(0,1)} |F_k|^q\right)^{1/q},
\| g_k\|_{C^\rho(B(0,1))} \right\} \le 1,
\end{equation}
and
\begin{equation}\label{3.2-6}
 \left(\average_{B(0, \theta)}
\big|u_k -\average_{B(0, \theta)} u_k
-(P_j^\beta +\varep_k \chi_j^{k\beta} (x/\varep_k))
 \average_{B(0, \theta)} \frac{\partial u^\beta_k}{\partial x_j}\big |^2\, dx \right)^{1/2}
 >\theta^{1+\sigma},
 \end{equation}
 where $(\chi_j^{k\beta}) $ denotes the correctors
 for the Stokes systems with coefficient matrices $A^k (x/\varep)$. 
 Note that by (\ref{3.2-5}) and Cacciopoli's inequality (\ref{Ca}),
  the sequence $\{ u_k\}$ is bounded in $H^1(B(0,1/2); \mathbb{R}^d)$.
 Thus, by passing to a subsequence, we may assume that
 $u_k \rightharpoonup u_0$ weakly in $L^2(B(0,1); \mathbb{R}^d)$ and
 $u_k \rightharpoonup u_0$ weakly in $H^1(B(0,1/2); \mathbb{R}^d)$.
 Similarly, in view of (\ref{3.2-5}), by passing to subsequences, we may assume that
 $g_k \to g_0$ in $L^\infty(B(0,1))$ and
 $F_k \rightharpoonup F_0$ weakly in  $L^q(B(0,1); \mathbb{R}^d)$.
 Since $\widehat{A^k}$ satisfies the ellipticity condition (\ref{ellipticity-1}),
 we may further assume that $\widehat{A^k} \to A^0$ for some constant matrix
 $A^0$ satisfying (\ref{ellipticity-1}).
 
 Since $\varep_k \chi_j^{k\beta} (x/\varep_k) \to 0$ strongly in $L^2(B(0,1); \mathbb{R}^d)$,
 by taking the limit in (\ref{3.2-6}), we obtain
 \begin{equation}\label{3.2-7}
\left(\average_{B(0, \theta)}
\big|u_0 -\average_{B(0, \theta)} u_0
-P_j^\beta 
 \average_{B(0, \theta)} \frac{\partial u^\beta_0}{\partial x_j}\big |^2\, dx \right)^{1/2}
 \ge \theta^{1+\sigma}.
 \end{equation}
 Also observe that (\ref{3.2-5}) implies 
 \begin{equation}\label{3.2-8}
 \max \left\{ \left(\average_{B(0,1)} |u_0|^2 \right)^{1/2},
\left(\average_{B(0,1)} |F_0|^q\right)^{1/q},
\| g_0\|_{C^\rho(B(0,1))} \right\} \le 1.
\end{equation}

Finally, we note that 
$$
\aligned
\| p_k -\average_{B(0,1/2)} p_k \|_{L^2(B(0,1/2))}
& \le C \, \| \nabla p_k \|_{H^{-1} (B(0,1/2))}\\
&\le C \Big\{ \|\nabla u_k\|_{L^2(B(0,1/2))} +\| F_k\|_{H^{-1}(B(0,1/2))}\Big\}\\
&\le C,
\endaligned
$$
where the first inequality holds for any $p_k\in L^2(B(0,1/2))$,
and we have used the first equation in (\ref{3.2-4}) for the second inequality and
 Cacciopoli's inequality for the third.
Clearly, we may assume $\int_{B(0, 1/2)} p_k =0$ by subtracting a constant.
Thus, by passing to a subsequence, we may assume that 
$p_k\rightharpoonup p_0$ weakly in $L^2(B(0,1/2))$.
This, together with convergence of $u_k$, $F_k$, $g_k$, and $\widehat{A^k}$,
 allows us to apply Theorem \ref{compactness-theorem} to conclude that
$-\text{\rm div} \big (A^0\nabla u_0) +\nabla p_0 =F_0$ and
$\text{\rm div} (u_0)=g_0$ in $B(0,1/2)$.
As a result, in view of (\ref{3.2-2}), (\ref{3.2-7}) and (\ref{3.2-8}),
we obtain
$$
\aligned
\theta^{1+\sigma}
 &\le C_0 \, \theta^{1+\rho}
\max \left\{ \left(\average_{B(0,1/2)} |u_0|^2\right)^{1/2},
\left(\average_{B(0,1/2)} |F_0|^q\right)^{1/q},
\| g_0\|_{C^\rho (B(0,1/2))} \right\} \\
&\le 2^dC_0 \theta^{1+\rho},
\endaligned
$$
which contradicts (\ref{3.2-4-1}).
This completes the proof.
\end{proof}

\begin{remark}\label{remark-3.1}
{\rm
It is easy to see that estimate (\ref{3.2-0}) continues to hold if we replace $\average_{B(0, \theta)} u_\varep$
by the average
$$
\average_{B(0,\theta)}
\bigg[ u_\varep -\big( P_j^\beta+\varep\chi_j^\beta (x/\varep) \big) \average_{B(0, \theta)} \frac{\partial u_\varep}
{\partial x_j}\bigg] \, dx.
$$
This will be used in the next lemma.
}
\end{remark}

\begin{lemma}\label{step-2}
Let $0<\sigma<\rho<1$ and $\rho=1-\frac{d}{q}$.
Let $(\varep_0, \theta)$ be the constants given by Lemma \ref{step-1}.
Suppose that $0<\varep<\theta^{k-1} \varep_0$ for some $k\ge 1$, and
\begin{equation}\label{3.3-0}
\mathcal{L}_\varep (u_\varep) +\nabla p_\varep =F \quad \text{ and } \quad
\text{\rm div} (u_\varep) =g \quad \text{ in } B(0, 1).
\end{equation}
Then there exist constants  $E(\varep, \ell)
=(E_j^\beta (\varep, \ell) ) \in \mathbb{R}^{d\times d}$ for $1\le \ell\le k$,
such that
\begin{equation}\label{3.3-1}
\aligned
 &\left(\average_{B(0, \theta^\ell)} \bigg| u_\varep -
  \big( P_j^\beta +\varep \chi_j^\beta (x/\varep) \big) E_j^\beta (\varep, \ell)
  -\average_{B(0, \theta^\ell)}
  \big[ u_\varep - \big( P_j^\beta +\varep\chi_j^\beta (x/\varep)\big) E_j^\beta (\varep, \ell) \big] \bigg|^2
\,  \right)^{1/2}\\
&\qquad
\le \theta^{\ell (1+\sigma)}
\max\left\{ \left(\average_{B(0,1)} |u_\varep|^2\right)^{1/2},
\left(\average_{B(0,1)} |F|^q\right)^{1/q},
\| g\|_{C^\rho (B(0,1))} \right\}.
\endaligned
\end{equation}
Moreover, the constants  $E (\varep, \ell )$
satisfy
\begin{equation}\label{3.3-2}
 |E(\varep, \ell)| \le 
C\max \left\{ \left(\average_{B(0,1)} |u_\varep|^2\right)^{1/2},
\left(\average_{B(0,1)} |F|^q\right)^{1/q},
\| g\|_{C^\rho (B(0,1))} \right\},
\end{equation}
\begin{equation}\label{3.3-3}
\aligned
&|E(\varep, \ell+1)-E(\varep, \ell)|\\
&\qquad \le C\, \theta^{\ell\sigma}\max \left\{ \left(\average_{B(0,1)} |u_\varep|^2\right)^{1/2},
\left(\average_{B(0,1)} |F|^q\right)^{1/q},
\| g\|_{C^\rho (B(0,1))} \right\},
\endaligned
\end{equation}
where $C$ depends only on $d$, $\mu$, $\sigma$ and $\rho$,
and
\begin{equation}\label{3.3-4}
\sum_{j=1}^d E_j^j (\varep, \ell)
=\average_{B(0, \theta^\ell)} g.
\end{equation}
\end{lemma}

\begin{proof}
The lemma is proved by an induction argument on $\ell$.
The case $\ell=1$ follows directly from Lemma \ref{step-1}, with
$$
E_j^\beta (\varep, 1) =\average_{B(0, \theta)} \frac{\partial u_\varep^\beta}{\partial x_j}
$$
(see Remark \ref{remark-3.1}).
Suppose that the desired constants exist for all positive integers up to some $\ell$, where $1\le \ell\le k-1$.
To construct $E(\varep, \ell +1)$, we consider
$$
\aligned
w(x)=& u_\varep (\theta^\ell x) - \left\{ P_j^\beta (\theta^\ell x)
+\varep  \chi_j^\beta (\theta^\ell x/\varep) \right\}
E_j^\beta (\varep, \ell)\\
&\qquad\qquad
-\average_{B(0, \theta^\ell)}
\bigg[ u_\varep -\big( P_j^\beta +\varep \chi_j^\beta (x/\varep) \big) E_j^\beta (\varep, \ell) \bigg].
\endaligned
$$
Note that by the rescaling property of the Stokes system,
\begin{equation}\label{3.3-5}
\left\{
\aligned
& \mathcal{L}_{\frac{\varep}{\theta^\ell}} (w) +\nabla \left\{ \theta^\ell p_\varep (\theta^\ell x) -
\theta^\ell \pi_j^\beta (\theta^\ell x/\varep) E_j^\beta (\varep, \ell) \right\}   = \theta^{2\ell} F(\theta^\ell x),\\
& \text{\rm div} (w) = \theta^\ell g(\theta^\ell x) -\theta^\ell \sum_{j=1}^d E_j^j (\varep, \ell),
\endaligned
\right.
\end{equation}
in $B(0,1)$, 
where $\pi_j^\beta $ is defined by (\ref{corrector}). 
Since $(\varep/\theta^\ell) \le ({\varep}/{\theta^{k-1}}) < \varep_0$,
we may apply Lemma \ref{step-1} to obtain
\begin{equation}\label{3.3-6}
\aligned
& \bigg(\average_{B(0, \theta)}
\bigg|w -\big(P_j^\beta +\theta^{-\ell} \varep \chi_j^\beta (\theta^\ell x/\varep)\big)
 \average_{B(0, \theta)} \frac{\partial w^\beta}{\partial x_j}\\
&\qquad\qquad\qquad 
 -\average_{B(0, \theta)}
 \bigg[w -\big(P_j^\beta +\theta^{-\ell} \varep \chi_j^\beta (\theta^\ell x/\varep)\big)
 \average_{B(0, \theta)} \frac{\partial w^\beta}{\partial x_j}\bigg] \bigg|^2
\, dx \bigg)^{1/2}\\
& \le \theta^{1+\sigma}
\max \left\{  \left(\average_{B(0,1)} |w|^2\right)^{1/2},
\left(\average_{B(0,1)} |F_\ell |^q dx\right)^{1/q},
\| \text{\rm div} (w) \|_{C^\rho (B(0,1))} \right\},
\endaligned
\end{equation}
where $F_\ell (x)=\theta^{2\ell} F(\theta^\ell x)$.

We now estimate the right hand side of (\ref{3.3-6}).
Observe that by the induction assumption,
\begin{equation}\label{3.3-7}
\aligned
&\left(\average_{B(0,1)} |w|^2\right)^{1/2}\\
&
\le \theta^{\ell (1+\sigma)}
\max\left\{ \left(\average_{B(0,1)} |u_\varep|^2\right)^{1/2},
\left(\average_{B(0,1)} |F|^q\right)^{1/q},
\| g\|_{C^\rho (B(0,1))} \right\}.
\endaligned
\end{equation}
Also note that since $0<\rho=1-\frac{d}{q}$,
$$
\left(\average_{B(0,1)} |\theta^{2\ell} F (\theta^\ell x)|^q dx\right)^{1/q}
\le \theta^{\ell (1+\rho)}
\left(\average_{B(0,1)} |F|^q\right)^{1/q}.
$$
In view of (\ref{3.3-5}) and (\ref{3.3-4}), we have
$$
\text{div} (w)=\theta^\ell \left\{ g(\theta^\ell x) -\average_{B(0, \theta^\ell)} g \right\},
$$
which gives
$$
\| \text{div} (w)\|_{C^{\rho} (B(0,1))}
\le \theta^{\ell (1+\rho)} \| g\|_{C^\rho(B(0,1))}.
$$
Thus we have proved that the right hand side of (\ref{3.3-6}) is bounded by
$$
\theta^{(\ell+1) (1+\sigma)}
\max\left\{ \left(\average_{B(0,1)} |u_\varep|^2\right)^{1/2},
\left(\average_{B(0,1)} |F|^q\right)^{1/q},
\| g\|_{C^\rho (B(0,1))} \right\}.
$$

Finally, we note that the left hand side of (\ref{3.3-6}) may be written as
$$
\aligned
\bigg(\average_{B(0, \theta^{\ell+1})} \bigg|u_\varep -
&  \big( P_j^\beta +\varep \chi_j^\beta (x/\varep) \big) E_j^\beta(\varep, \ell+1) \\
& -\average_{B(0, \theta^{\ell+1})}
 \bigg[
 u_\varep -
 \big( P_j^\beta +\varep \chi_j^\beta (x/\varep) \big) E_j^\beta(\varep, \ell+1) \bigg]
 \bigg|^2\, dx \bigg)^{1/2}
\endaligned
$$
with
\begin{equation}\label{3.3-10}
E_j^\beta (\varep, \ell +1)  =E_j^\beta (\varep, \ell) +\theta^{-\ell} \average_{B(0, \theta)} \frac{\partial w^\beta}{\partial x_j}.
\end{equation}
Observe that by Cacciopoli's inequality (\ref{Ca}),
$$
\aligned
& |E (\varep, \ell+1)-E(\varep, \ell)|
\le \theta^{-\ell} \left(\average_{B(0, \theta)} |\nabla w|^2\right)^{1/2}\\
&\le C\theta^{-\ell}
\max \left\{ \left(\average_{B(0,1)} |w|^2\right)^{1/2},
\left(\average_{B(0,1)} |\theta^{2\ell} F(\theta^{2\ell} x)|^2\right)^{1/2},
\left(\average_{B(0,1)} |\text{\rm div} (w)|^2\right)^{1/2}\right\}\\
&\le C\theta^{\ell\sigma}
\max
\left\{\left(\average_{B(0,1)} |u_\varep|^2\right)^{1/2},
\left(\average_{B(0,1)} |F|^q\right)^{1/q},
\| g\|_{C^\rho (B(0,1))} \right\},
\endaligned
$$
where we have used the estimates for the right hand side of (\ref{3.3-6}) for the last inequality.
This, together with the estimate of $E(\varep, 1)$,
gives (\ref{3.3-2}) and (\ref{3.3-3}).
To see (\ref{3.3-4}), we note that by (\ref{3.3-10}) and (\ref{3.3-5}),
$$
\aligned
\sum_{j=1}^d E_j^j (\varep, \ell +1)
&=\sum_{j=1}^d E_j^j (\varep, \ell) +\theta^{-\ell} \average_{B(0,\theta)} \text{\rm div} (w)
=\average_{B(0, \theta)} g(\theta^\ell x)\, dx \\
&=\average_{B(0, \theta^{\ell+1})} g,
\endaligned
$$
This completes the proof.
\end{proof}

The following theorem may be viewed as the Lipschitz estimate for $u_\varep$, 
down to the scale $\varep$.
We use $[g]_{C^{0, \rho}(E)}$ to denote the semi-norm
$$
[g]_{C^{0, \rho}(E)}
=\sup \left\{ \frac{|g(x)-g(y)|}{|x-y|^\rho}: \ x,y\in E \text{ and } x\neq y \right\}.
$$

\begin{theorem}\label{theorem-3.2}
Suppose that $A(y)$ satisfies the ellipticity condition (\ref{ellipticity})
and is 1-periodic.
Let $(u_\varep, p_\varep)$ be a weak solution of
\begin{equation}\label{3.4-0}
\mathcal{L}_\varep (u_\varep) +\nabla p_\varep =F \quad \text{ and } \quad \text{\rm div} (u_\varep) =g
\end{equation}
in $B(x_0, R)$ for some $x_0\in \mathbb{R}^d$ and $R>2\varep$.
Then, if $\varep\le r\le (R/2)$,
\begin{equation}\label{3.4-1}
\aligned
\left(\average_{B(x_0, r)} |\nabla u_\varep|^2 \right)^{1/2}
 & \le C \bigg\{
\frac{1}{R} \left(\average_{B(x_0, R)} |u_\varep|^2\right)^{1/2}
+R \left(\average_{B(x_0, R)} |F|^q\right)^{1/q}\\
& \qquad\qquad\qquad
+ \| g\|_{L^\infty(B(x_0, R))}
+ R^\rho [g]_{C^{0, \rho} (B(x_0, R))} \bigg\},
\endaligned
\end{equation}
where $\rho \in (0, 1)$, $\rho=1-\frac{d}{q}$, and
$C$ depends only on $d$, $\mu$, and $\rho$.
\end{theorem}

\begin{proof}
By covering $B(x_0, r)$ with balls of radius $\varep$ we only need to consider the case $r=\varep$.
By translation and dilation we may further assume that $x_0=0$ and $R=1$.
Thus we need to show that if $0<\varep\le (1/2)$,
\begin{equation}\label{3.4-3}
\left(\average_{B(0, \varep)}
|\nabla u_\varep|^2\right)^{1/2}
\le C\left\{ \left(\average_{B(0,1)} |u_\varep|^2\right)^{1/2}
+\left(\average_{B(0,1)} |F|^q\right)^{1/q}
+\| g\|_{C^\rho (B(0,1))} \right\}.
\end{equation}
We will see that this follows readily from Lemma \ref{step-2}.

Indeed, let $(\varep_0, \theta)$ be given by Lemma \ref{step-1}.
The case $\theta \varep_0\le \varep\le (1/2)$ follows directly from  Cacciopoli's inequality.
Suppose $0<\varep<\theta \varep_0$.
Choose $k\ge 2$ so that $\theta^k \varep_0 \le \varep< \theta^{k-1} \varep_0$.
It follows from Lemma \ref{step-2} that
\begin{equation}\label{3.4-4}
\aligned
& \left(\average_{B(0, \theta^{k-1})} \big|u_\varep -\average_{B(0, \theta^{k-1})} u_\varep\big|^2 \right)^{1/2}\\
&\qquad 
\le C\left\{ \left(\average_{B(0,1)} |u_\varep|^2\right)^{1/2}
+\left(\average_{B(0,1)} |F|^q\right)^{1/q}
+\| g\|_{C^\rho (B(0,1))} \right\}.
\endaligned
\end{equation}
This, together with the Cacciopoli's inequality, implies that
$$
\left(\average_{B(0, \theta^{k-1})}
|\nabla u_\varep|^2\right)^{1/2}
\le C\left\{ \left(\average_{B(0,1)} |u_\varep|^2\right)^{1/2}
+\left(\average_{B(0,1)} |F|^q\right)^{1/q}
+\| g\|_{C^\rho (B(0,1))} \right\},
$$
from which the estimate (\ref{3.4-3}) follows.
\end{proof}



\section{A Liouville property for Stokes systems}
\setcounter{equation}{0}

In this section we prove a Liouville property for global solutions of the Stokes systems
with periodic coefficients. We refer the reader to \cite{AL-liouv} for the case of 
the elliptic systems $\mathcal{L}_1 (u)=0$ (also see \cite{Struwe-1992, Kuchment-2007} and their references
for related work). The Liouville property for  Stokes systems with constant coefficients is well known;
however, the authors are not aware of any previous work on the Liouville property for Stokes systems
with variable coefficients.

\begin{theorem}\label{L-theorem}
Suppose that $A(y)$ satisfies the ellipticity condition (\ref{ellipticity}) and is 1-periodic.
Let $ (u, p)\in H^1_{\text loc} (\mathbb{R}^d; \mathbb{R}^d) \times L^2_{\text loc} (\mathbb{R}^d)$
be a weak solution of
\begin{equation} \label{4.1-0}
\mathcal{L}_1 (u) +\nabla p =0 \quad \text{ and } \quad \text{\rm div} (u) =g
\end{equation}
in $\mathbb{R}^d$, where $g$ is constant. Assume that
\begin{equation}\label{4.1-1}
\left(\average_{B(0, R)} |u|^2\right)^{1/2} \le C_u \, R^{1+\sigma}
\end{equation}
for some $C_u>0$, $\sigma \in (0,1)$, and for all $R>1$.
Then 
\begin{equation}\label{4.1-2}
\left\{
\aligned
u(x)  &= H+ \big( P_j^\beta (x)+\chi_j^\beta (x)\big) E_j^\beta,\\
p(x) &= \widetilde{H} +\pi_j^\beta (x) E_j^\beta
\endaligned
\right.
\end{equation}
for some constants $H\in \mathbb{R}^d$, $\widetilde{H}\in \mathbb{R}$, and  $E =(E_j^\beta) \in \mathbb{R}^{d\times d}$.
In particular, the space of functions $(u,p)$ that satisfy (\ref{4.1-0}) and (\ref{4.1-1})
is of dimension $d^2+d+1$.
\end{theorem}

\begin{proof}
Fix $\sigma_1\in (\sigma,1)$. Let
$(\varep_0, \theta)$ be the constants given by Lemma \ref{step-1} for $0<\sigma_1<\rho<1$.
Suppose that $(u,p)$ is a solution of (\ref{4.1-0}) in $\mathbb{R}^d$ for some constant $g$.
Let $u_\varep(x) =u(x/\varep )$ and $p_\varep (x)=\varep^{-1} p( x/\varep )$.
Then $\mathcal{L}_\varep (u_\varep) +\nabla p_\varep=0$ and
$\text{div}(u_\varep) (x) =\varep^{-1} g$ in $B(0,1)$.
It follows from Lemma \ref{step-2} that 
if $0<\varep<\theta^{k-1} \varep_0$ for some $k\ge 1$, then
$$
\aligned
 \inf_{\substack{E=(E_j^\beta)\in \mathbb{R}^{d\times d}\\ H\in \mathbb{R}^d}}
& \left(\average_{B(0, \theta^\ell)} 
\big |u_\varep - \big( P_j^\beta  +\varep \chi_j^\beta (x/\varep) \big) E_j^\beta -H\big |^2 \right)^{1/2}\\
&\le \theta^{\ell (1+\sigma_1)}
\max \left\{ \left(\average_{B(0,1)} |u_\varep|^2\right)^{1/2},
\varep^{-1} | g|  \right\},
\endaligned
$$
where $1\le \ell \le k$.
By a change of variables this gives
\begin{equation}\label{4.1-3}
\aligned
 \inf_{\substack{E=(E_j^\beta)\in \mathbb{R}^{d\times d}\\ H\in \mathbb{R}^d}}
& \left(\average_{B(0, \varep^{-1} \theta^\ell)} 
\big |u - \big( P_j^\beta  + \chi_j^\beta (x ) \big) E_j^\beta -H\big |^2 \right)^{1/2}\\
&\le \theta^{\ell (1+\sigma_1)}
\max \left\{ \left(\average_{B(0,\varep^{-1} )} |u|^2\right)^{1/2},
\varep^{-1} | g|  \right\},
\endaligned
\end{equation}
where $0<\varep<\theta^{k-1} \varep_0$ for some $k\ge 1$ and $1\le \ell\le k$.

Now, suppose that $u$ satisfies the growth condition (\ref{4.1-1}).
 For any $m\ge 1$ such that $\theta^{m+1}<\varep_0$, let $\varep=\theta^{m+\ell}$, where $\ell>1$.
It follows from (\ref{4.1-3}) and (\ref{4.1-1})  that
\begin{equation}\label{4.1-4}
\aligned
\inf_{\substack{ E=(E_j^\beta)\in \mathbb{R}^{d\times d}\\ H\in \mathbb{R}^d}}
 &\left(\average_{B(0, \theta^{-m})}\left |u -\big( P_j^\beta +\chi_j^\beta  \big)E_j^\beta 
-H\right |^2\right)^{1/2}\\
& \le \theta^{\ell (1+\sigma_1)} 
\max \Big\{ C (\varep^{-1})^{1+\sigma}, \varep^{-1} |g|\Big\}\\
&= \theta^{\ell (1+\sigma_1)} 
\max \Big \{ C \theta^{-(m+\ell) (1+\sigma)}, \theta^{-(m+\ell)}  |g|\Big\},
\endaligned
\end{equation}
for some constant $C$ independent of $m$ and $\ell $.
Since $\sigma_1>\sigma$, we may fix $m$ and let $\ell \to \infty$ in (\ref{4.1-4}) to conclude that
the left hand side of (\ref{4.1-4}) is zero.
Thus, for each $m$ large, there exist constants $H^m \in \mathbb{R}^d$ and $E^m = (E_j^{m\beta})\in \mathbb{R}^{d\times d}$
such that
$$
u(x) =H^m + \big( P_j^\beta  (x)+\chi_j^\beta (x) \big) E_j^{m\beta} \qquad \text{ in } B(0, \theta^{-m}).
$$

Finally, we observe that $\nabla u= (\nabla P_j^\beta +\nabla \chi_j^{\beta} )E_j^{m\beta}$ 
and since $\int_Y \nabla \chi_j^\beta=0$,
$$
\int_Y \nabla u=\int_Y \nabla P_j^\beta \cdot E_j^{m\beta}.
$$
This implies that $E_j^{m\beta}=E_j^{n \beta}$ for any $m, n$ large; and as a consequence,
we also obtain $H^m =H^n$ for any $m, n$ large.
Thus we have proved that (\ref{4.1-2}) holds for some $H\in \mathbb{R}^d$ and $E=(E_j^\beta)\in 
\mathbb{R}^{d\times d}$.
Note that if $H+( P_j^\beta +\chi_j^\beta ) E_j^\beta =0$ in $\mathbb{R}^d$,
then $\int_Y  \nabla P_j^\beta \cdot E_j^\beta=0$.
It follows that $E_j^\beta=0$ and hence, $H=0$.
This shows that  the space of functions $(u, p)$ that satisfy (\ref{4.1-0})-(\ref{4.1-1})
is of dimension $d^2+d+1$.
\end{proof}

\begin{remark}\label{remark-4.1}
{\rm
Suppose that $(u, p)$ satisfies (\ref{4.1-0}) in $\mathbb{R}^d$ for some constant $g$ and that
\begin{equation}\label{4.2-1}
\left(\average_{B(0, R)} |u|^2\right)^{1/2} \le C_u \, R^{\sigma}
\end{equation}
for some $C_u>0$, $\sigma \in (0,1)$, and for all $R>1$.
It follows from Theorem \ref{L-theorem} that $(u, p)$ must be constant.
}
\end{remark}

\begin{remark}
{\rm 
One may use the results in Theorem \ref{L-theorem} and a line of argument used in \cite{Struwe-1992}
to characterize all solutions of (\ref{4.1-0}) in $\mathbb{R}^d$ that satisfy the growth condition
\begin{equation}\label{4.3-1}
\left(\average_{B(0, R)} |u|^2\right)^{1/2} \le C_u \, R^{N+\sigma}
\end{equation}
for some $C_u>0$, integer $N\ge 2$, $\sigma \in (0,1)$, and for all $R>1$.
In particular, by using the difference operator $\Delta_i \phi
= \phi (x+e_i) -\phi (x)$  repeatedly,
one may deduce from the observation in Remark \ref{remark-4.1} that
$$
u^\alpha (x)
=\sum_{|\nu|=N} E(\nu, \alpha) x^\nu + \sum_{0\le |\nu|\le N-1} w_{\nu, \alpha} (x) x^\nu,
$$
where  $E(\nu, \alpha)$ is constant and $w_{\nu, \alpha} (x)$ is 1-periodic. Here
$\nu=(\nu_1, \nu_2, \dots, \nu_d)$ is a multi-index and 
$x^\nu=x_1^{\nu_1}x_2^{\nu_2} \cdots x_d^{\nu_d}$.
We will pursue this line of research elsewhere.
}
\end{remark}



\section{$L^\infty$ estimates for $p_\varep$ and proof of Theorem \ref{main-theorem-1}}
\setcounter{equation}{0}

In this section we prove an $L^\infty$ estimate for $p_\varep$, down to the scale $\varep$.
We also give the proof of Theorem \ref{main-theorem-1} and Corollary \ref{corollary-1}.

\begin{theorem}\label{theorem-5.1}
Suppose that $A(y)$ satisfies the ellipticity condition (\ref{ellipticity}) and
is 1-periodic. 
Let $(u_\varep, p_\varep)$ be a weak solution of
\begin{equation}\label{5.1-0}
\mathcal{L}_\varep (u_\varep) +\nabla p_\varep =F \quad \text{ and } \quad \text{\rm div} (u_\varep) =g
\end{equation}
in $B(x_0, R)$ for some $x_0\in \mathbb{R}^d$ and $R>\varep$.
Then, if $\varep\le r< R$,
\begin{equation}\label{5.1-1}
\aligned
 \left(\average_{B(x_0, r)} |p_\varep -\average_{B(x_0, R)} p_\varep|^2 \right)^{1/2}
&  \le C \bigg\{
 \left(\average_{B(x_0, R)} |\nabla u_\varep|^2\right)^{1/2}
+R \left(\average_{B(x_0, R)} |F|^q\right)^{1/q}\\
&\qquad\qquad
+ \| g\|_{L^\infty(B(x_0, R))}
+ R^\rho [g]_{C^{0, \rho} (B(x_0, R))} \bigg\},
\endaligned
\end{equation}
where $\rho \in (0, 1)$, $\rho=1-\frac{d}{q}$, and
$C$ depends only on $d$, $\mu$ and $\rho$.
\end{theorem}

\begin{proof}
By translation and dilation we may assume that $x_0=0$ and $R=1$.
Note that
\begin{equation}\label{5.1-2}
\aligned
\|p_\varep -\average_{B(0,r)} p_\varep\|_{L^2(B(0,r))}
&\le C\, \| \nabla p_\varep\|_{H^{-1} (B(0,r))}\\
& \le C \Big\{ \|\nabla u_\varep\|_{L^2(B(0,r))}
+\| F\|_{H^{-1}(B(0,r))}\Big\},
\endaligned
\end{equation}
where we have used the first equation in (\ref{5.1-0}) for the second inequality.
Thus, in view of Theorem \ref{theorem-3.2}, it suffices to show that
$\big|\average_{B(0,r)}  p_\varep-\average_{B(0,1)} p_\varep\big|$ 
is bounded by the right hand side of
(\ref{5.1-1}). This will be done by using the $C^{1, \sigma}$ estimate for $u_\varep$ down to the scale $\varep$ in 
Lemma \ref{step-2}.

Let $(\theta, \varep_0)$ be the constants given by Lemma \ref{step-1}.
By (\ref{5.1-2}) we may assume that $0<\varep\le r<\varep_0$.
Let $\theta^k \varep_0 \le \varep<\theta^{k-1} \varep_0$ and
$\theta^t \varep_0\le r < \theta^{t-1}\varep_0$ for some $1\le t \le k$.
The terms $\average_{B(0,r)} p_\varep -\average_{B(0, \theta^t)} p_\varep$
and $\average_{B(0,1)} p_\varep -\average_{B(0,\theta)} p_\varep$
can be handled by using (\ref{5.1-2}).
To deal with $\average_{B(0,\theta^t)} p_\varep -\average_{B(0,\theta)} p_\varep$, we write
\begin{equation}\label{5.1-3}
\int_{B(0, \theta^t )} p_\varep -\int_{B(0, \theta)} p_\varep
=\sum_{\ell=1}^{t-1} \left\{ \average_{B(0, \theta^{\ell+1})} p_\varep -\average_{B(0, \theta^{\ell})} p_\varep\right\}.
\end{equation}
Let
$$
\aligned
v_\ell=u_\varep (x) &-\big( P_j^\beta (x) +\varep \chi_j^\beta (x/\varep) \big) E_j^\beta (\varep, \ell)\\
&-\average_{B(0,\theta^\ell)} 
\left\{ u_\varep (x) -\big( P_j^\beta (x) +\varep \chi_j^\beta (x/\varep) \big) E_j^\beta (\varep, \ell)\right\} dx,
\endaligned
$$
where $E(\varep, \ell)= (E_j^\beta(\varep, \ell) )\in \mathbb{R}^{d\times d}$ are constants given by
Lemma  \ref{step-2}.
Note that by Lemma \ref{step-2},
\begin{equation}\label{5.1-3-1}
\aligned
& \left(\average_{B(0, \theta^\ell)} |v_\ell|^2 \right)^{1/2}\\
&\le \theta^{\ell(1+\sigma)}
\max \left\{ \left(\average_{B(0,1)} |u_\varep|^2\right)^{1/2},
\left(\average_{B(0,1)} |F|^q\right)^{1/q},
\| g\|_{C^\rho(B(0,1))} \right\},
\endaligned
\end{equation}
where $0<\sigma<\rho<1$, and
\begin{equation}\label{5.1-4}
\left\{\aligned
&\mathcal{L}_\varep (v_\ell)+\nabla \left\{ p_\varep -\pi_j^\beta (x/\varep) E_j^\beta (\varep, \ell) \right\}=F,\\
&\text{\rm div} (v_\ell) =g -\average_{B(0, \theta^\ell)} g,
\endaligned
\right.
\end{equation}
in $B(0,1)$.
Observe that for any $H\in \mathbb{R}$,
\begin{equation}\label{5.1-5}
\aligned
&\left|\average_{B(0, \theta^{\ell+1})} p_\varep -\average_{B(0, \theta^\ell)} p_\varep \right|\\
&\le
\left|\average_{B(0, \theta^{\ell+1})}
\big[ p_\varep -H -\pi_j^\beta (x/\varep) E_j^\beta (\varep, \ell) \big] dx\right|\\
&\qquad\qquad\qquad
+\left|\average_{B(0, \theta^\ell)} \big[ p_\varep -H-\pi_j^\beta (x/\varep) E_j^\beta (\varep, \ell) \big] dx \right|\\
&\qquad\qquad\qquad
+ |E_j^\beta (\varep, \ell)|
\left|\average_{B(0, \theta^{\ell+1})}
\pi_j^\beta (x/\varep) dx
-\average_{B(0, \theta^{\ell})}
\pi_j^\beta (x/\varep) dx \right|.
\endaligned
\end{equation}
Choose
$$
H=\average_{B(0, \theta^\ell)} \big[ p_\varep -\pi_j^\beta (x/\varep) E_j^\beta (\varep, \ell) \big] dx
$$
so that the second term in the right hand side of (\ref{5.1-5}) equals to zero.
Using (\ref{5.1-2}), (\ref{5.1-4}), Cacciopoli's inequality and (\ref{5.1-3-1}),
we see that the first term in the right hand side of (\ref{5.1-5}) is bounded by
$$
\aligned
& C \left( \average_{B(0, \theta^\ell)} \big| p_\varep -H-\pi_j^\beta (x/\varep) E_j^\beta (\varep, \ell) \big|^2 dx\right)^{1/2}\\
&\qquad\qquad \le C \theta^{-d\ell /2}
\Big\{ \|\nabla v_\ell\|_{L^2(B(0, \theta^{\ell}))}
+\|F\|_{H^{-1}(B(0, \theta^\ell))} \Big\}\\
& \qquad \qquad 
\le C\, \theta^{\ell\sigma}
\max \left\{ \left(\average_{B(0,1)} |u_\varep|^2\right)^{1/2},
\left(\average_{B(0,1)} |F|^q\right)^{1/q},
\| g\|_{C^\rho(B(0,1))} \right\},
\endaligned
$$
where we also used $q>d$, $0<\sigma<\rho=1-\frac{d}{q}$, and
$$
\aligned
\| F\|_{H^{-1} (B(0, \theta^\ell))}  &
\le C |B(0, \theta^\ell)|^{\frac{1}{2} +\frac{1}{d} }
\left(\average_{B(0, \theta^\ell)} |F|^q\right)^{1/q}\\
&\le C \theta^{\ell (\frac{d}{2} +\rho)} 
\left(\average_{B(0,1)} |F|^q\right)^{1/q}.
\endaligned
$$

Finally, we note that since $\pi_j^\beta$ is 1-periodic,
\begin{equation}\label{5.1-6}
\aligned
&\left|\average_{B(0, \theta^{\ell+1})}
\pi_j^\beta (x/\varep) dx
-\average_{B(0, \theta^{\ell})}
\pi_j^\beta (x/\varep) dx \right|\\
 &=\left|\average_{B(0, \varep^{-1} \theta^{\ell+1})} \pi_j^\beta -\langle \pi_j^\beta\rangle \right| +
\left|\average_{B(0, \varep^{-1} \theta^\ell)} \pi_j^\beta -\langle \pi_j^\beta \rangle \right|\\
& \le C \varep \theta^{-\ell} \| \pi_j^\beta\|_{L^2(Y)}\\
&\le C \, \varep  \theta^{-\ell},
\endaligned
\end{equation}
where $\langle \pi_j^\beta\rangle=\average_Y \pi_j^\beta$.
This, together with the estimate of the first two terms in the right hand side of (\ref{5.1-5}),
shows that the left hand side of (\ref{5.1-3}) is bounded by
$$
\aligned
& C  \sum_{\ell=1}^{t-1} \big( \theta^{\ell \sigma}
+\varep  \theta^{-\ell} \big)
\max \left\{ \left(\average_{B(0,1)} |u_\varep|^2\right)^{1/2},
\left(\average_{B(0,1)} |F|^q\right)^{1/q},
\| g\|_{C^\rho(B(0,1))} \right\}\\
& 
\le C\max \left\{ \left(\average_{B(0,1)} |u_\varep|^2\right)^{1/2},
\left(\average_{B(0,1)} |F|^q\right)^{1/q},
\| g\|_{C^\rho(B(0,1))} \right\},
\endaligned
$$
This completes the proof.
\end{proof}

\begin{proof}[\bf Proof of Theorem \ref{main-theorem-1}]
The estimate for $\nabla u_\varep $ in (\ref{main-estimate-1}) is given by Theorem  \ref{theorem-3.2}, while
the estimate for $ p_\varep$  is contained in Theorem \ref{theorem-5.1}.
\end{proof}

\begin{proof}[\bf Proof of Corollary \ref{corollary-1}]
Under the H\"older continuous condition (\ref{holder}),
it is known that solutions of the Stokes systems are locally $C^{1, \alpha}$ for $\alpha<\lambda$
(see \cite{Gia-1982}).
In particular, it follows that if
$(u, p)$ is a weak solution of $-\text{\rm div}(A(x)\nabla u) +\nabla p =F$ 
and $\text{div} (u)=g$ in $B(y,1)$ for some $y\in \mathbb{R}^d$, then
\begin{equation}\label{5.2-1}
\aligned
&\| \nabla u\|_{L^\infty (B(y, 1/2))}
+\| p-\average_{B(y,1/2)} p\|_{L^\infty(B(y,1/2))}\\
&\le C\left\{
\left(\average_{B(y, 1)} |\nabla u|^2\right)^{1/2}
+ \left(\average_{B(y, 1)} |F|^q\right)^{1/q}
+\| g\|_{C^\rho(B(y, 1))} \right\},
\endaligned
\end{equation}
where $0<\rho< 1$, $\rho=1-\frac{d}{q}$, and the constant $C$ depends only on
$d$, $\mu$, $\rho$, and $(\lambda, \tau)$ in (\ref{holder}).

To prove (\ref{Lip-estimate}), by translation and dilation, we may assume that $x_0=0$ and $R=1$.
Now suppose $(u_\varep, p_\varep)$ is a weak solution of (\ref{Stokes}) in $B(0,1)$.
The estimate (\ref{Lip-estimate}) for the case $\varep\ge (1/8)$ follows directly from (\ref{5.2-1}), as the matrix 
$A(x/\varep)$ satisfies (\ref{holder}) uniformly in $\varep$.
For $0<\varep<(1/8)$, we use a blow-up argument and estimate
(\ref{5.2-1}) by considering $u (x)=\varep^{-1} u_\varep (\varep x)$ and $p(x)=p_\varep (\varep x)$.
This leads to
\begin{equation}\label{5.2-2}
\aligned
&\|\nabla u_\varep\|_{L^\infty(B(y, \varep))} +\| p_\varep-\average_{B(y, \varep)} p_\varep\|_{L^\infty(B(y, \varep))}\\
&\le C \left\{
\left(\average_{B(y, 2\varep)} |\nabla u_\varep|^2\right)^{1/2}
+\varep \left(\average_{B(y,2\varep )} |F|^q\right)^{1/q}
+\| g\|_{C^\rho(B(y,2\varep))} \right\},
\endaligned
\end{equation}
for any $y\in B(0, 1/2)$. In view of Theorem \ref{theorem-3.2} we obtain
\begin{equation}\label{5.2-3}
\aligned
& \|\nabla u_\varep\|_{L^\infty(B(0,1/2))}+\| p_\varep-\average_{B(y, \varep)} p_\varep\|_{L^\infty(B(y, \varep))}\\
&\le C
\left\{
\left(\average_{B(0, 1)} |\nabla u_\varep|^2\right)^{1/2}
+\left(\average_{B(0,1)} |F|^q\right)^{1/q}
+\| g\|_{C^\rho(B(0,1))} \right\}.
\endaligned
\end{equation}

Finally, we note that for any $y\in B(0,1/2)$,
\begin{equation*}
\aligned
& |p_\varep (y)-\average_{B(0, 1)} p_\varep |\\
&\le |p_\varep (y) -\average_{B(y, \varep)} p_\varep|
+|\average_{B(y, \varep)} p_\varep
-\average_{B(y, 1/2)} p_\varep|
+|\average_{B(y, 1/2)} p_\varep -\average_{B(0,1)} p_\varep|\\
&\le 
|p_\varep (y) -\average_{B(y, \varep)} p_\varep|
+\left(\average_{B(y, \varep)} |p_\varep -\average_{B(y, 1/2)} p_\varep|^2 \right)^{1/2}
+\left(\average_{B(0,1)} |p_\varep -\average_{B(0,1)} p_\varep|^2\right)^{1/2}\\
&\le C
\left\{
\left(\average_{B(0, 1)} |\nabla u_\varep|^2\right)^{1/2}
+\left(\average_{B(0,1)} |F|^q\right)^{1/q}
+\| g\|_{C^\rho(B(0,1))} \right\},
\endaligned
\end{equation*}
where we have used (\ref{5.2-2}), (\ref{5.2-3}), Theorem \ref{theorem-5.1}, and (\ref{5.1-2})
for the last inequality.
This completes the proof.
\end{proof}



\section{Boundary H\"older estimates and proof of Theorem \ref{main-theorem-2}}
\setcounter{equation}{0}

In this section we establish uniform boundary H\"older estimates for the Stokes system (\ref{Stokes})
in $C^1$ domains and give the proof of Theorem \ref{main-theorem-2}.

Let $\psi:\mathbb{R}^{d-1}\to \mathbb{R}$ be a $C^1$ function and
\begin{equation}\label{6-0}
\aligned
D_r &= D(r, \psi)
=\big\{ x=(x^\prime, x_d)\in \mathbb{R}^d: \, |x^\prime|<r \text{ and } \psi(x^\prime)<x_d<\psi(x^\prime) +10(M+1)r ) \big\},\\
\Delta_r & = \Delta (r, \psi)
=\big\{ x=(x^\prime, x_d)\in \mathbb{R}^d: \, |x^\prime|<r \text{ and } x_d=\psi(x^\prime) \big\}.
\endaligned
\end{equation}
We will always assume that $\psi (0)=0$ and
\begin{equation}\label{module}
\|\nabla \psi\|_\infty\le M \text{ and }
|\nabla \psi (x^\prime)-\nabla \psi (y^\prime)|\le \omega \big( |x^\prime-y^\prime|\big) \quad \text{ for any } x^\prime, y^\prime 
\in \mathbb{R}^{d-1},
\end{equation}
where $M>0$ is a fixed constant and
$\omega (r)$ is a fixed, nondecreasing continuous function on $[0, \infty)$ and $\omega (0)=0$.

\begin{theorem}\label{theorem-6.1}
Let $0<\rho, \eta<1$. 
Let $(u_\varepsilon, p_\varep)\in H^1(D_r;\mathbb{R}^d)\times L^2(D_r)$ be a weak solution of
\begin{equation}\label{6.1-0}
\left\{
\aligned
 \mathcal{L}_\varepsilon(u_\varepsilon) +\nabla p_\varep& =0&\quad &\text{ in } D_r,\\
 \text{ \rm div}(u_\varepsilon)  & = g &\quad &  \text{ in }D_r,\\
 u_\varepsilon  &= h &\quad  &\text{ on }\Delta_r
\endaligned
\right.
\end{equation}
for some $0<\varep<r<r_0$, where $g\in C^\eta(D_r)$, $h\in C^{0,1}(\Delta_r)$ and $h(0)=0$.
 Then for any $0<\varep\le t < r$,
\begin{equation}\label{6.1-1}
\aligned
&\left(\average_{D_t} |u_\varepsilon |^2\right)^{1/2}\\
& \le C\left(\frac{t}{r}\right)^\rho \left\{\left(\average_{D_r}|u_\varepsilon|^2\right)^{1/2}
+r \|g\|_{L^\infty(D_r)} +r^{1+\eta} [ g]_{C^{0, \eta}(D_r)}
+ r[h]_{C^{0,1}(\Delta_r)} \right\},
\endaligned
\end{equation}
where $C$ depends only on $d$, $\mu$, $\rho$, $\eta$, $r_0$, and $(M, \omega)$ in (\ref{module}).
\end{theorem}

It is not hard to see that Theorem \ref{main-theorem-2}
follows from Theorem \ref{theorem-6.1} and
the following boundary Cacciopoli's inequality whose proof may be found in \cite{Gia-1982}.

\begin{theorem}
Suppose that $A$ satisfies the ellipticity condition (\ref{ellipticity}).
Let $(u, p) \in H^1(D_r;\mathbb{R}^d)\times L^2(D_r)$
 be a weak solution of
$$
\left\{
\aligned
 -\text{\rm div} \big(A(x)\nabla u\big) +\nabla p & =F +\text{\rm div}(f) &\quad &  \text{ in } D_r,\\
 \text{\rm div} (u) & = g &\quad & \text{ in } D_r,\\
 u & =h &\quad & \text{ on } \Delta_r.
\endaligned
\right.
$$
Then
\begin{equation}\label{B-Cacciopoli}
\int_{D_{r/2}} |\nabla u|^2  \le C\left\{\frac{1}{r^2} \int_{D_r} |u|^2 
 +\int_{D_r} |f|^2 +\int_{D_r} |g|^2 + r^2 \int_{D_r} |F|^2  +\|h\|^2_{H^{1/2}(\Delta_r)} \right\},
\end{equation}
where $\mathit{C}$ depends only on $d$, $\mu$, and $M$.
\end{theorem}

To prove Theorem \ref{theorem-6.1}
we need an analogue of Theorem \ref{compactness-theorem}
in the presence of boundary.

\begin{lemma}\label{boundary-compactness-theorem}
Let $\{A^k(y)\}$ be a sequence of 1-periodic matrices satisfying the ellipticity condition(\ref{ellipticity}).
Let $D(k)= D(r, \psi_k)$ and $\Delta (k)=\Delta (r, \psi_k)$,
where $\{\psi_k\}$ is a sequence of $C^1$ functions satisfying $\psi_k(0)=0$ and (\ref{module}).
 Let $(u_k, p_k) \in H^1(D(k);\mathbb{R}^d))\times L^2(D(k))$ be a weak solution of
$$
\left \{
\aligned
 -\text{\rm div} \big(A^k(x/\varep_k)\nabla u_k \big) +\nabla p_k &=0 &\quad&\text{ in } D(k), \\
 \text{\rm div} (u_k)  &= g_k& \quad & \text{ in } D(k),\\
 u_k &=h_k &\quad & \text{ on }\Delta(k),
\endaligned
\right .
$$
where $\varepsilon_k \rightarrow 0$, $f_k(0)=0$, and
\begin{equation}\label{temp3.3}
\|u_k\|_{H^1(D(k))}+ \| p_k\|_{L^2(D(k))} +\|g_k\|_{C^{\eta}(D(k))}+ \|h_k\|_{C^{0,1}(\Delta(k))} \le C.
\end{equation}
Then there exist subsequences of $\{ A^k\}$, $\{ u_k\}$, $\{ p_k\}$,
$\{\psi_k\}$, $\{ g_k\}$ and $\{h_k\}$, which we will still denote by the same notation, 
and a constant matrix $A^0$ satisfying (\ref{ellipticity-1}),
a function $\psi_0$ satisfying $\psi_0(0)=0$ and
(\ref{module}), $u_0\in H^1(D(r,\psi_0);\mathbb{R}^d)$, $p_0\in L^2(D(r,\psi_0))$, $g_0\in C^{\eta}(D(r, \psi_0))$,
$h_0\in C^{0,r}(\Delta(r,\psi_0);\mathbb{R}^d)$ such that
\begin{equation}\label{temp3.4}
\left \{
\aligned
&\widehat{A^k} \to A^0,\\
& \psi_k (x^\prime) \rightarrow \psi_0 (x^\prime)
\text{ and } \nabla \psi_k (x^\prime) \to \nabla \psi_0 (x^\prime) \text{ uniformly for } |x'|<r,\\
& h_k(x',\psi_k(x')) \rightarrow h_0(x',\psi_0(x')) \text{ uniformly for }|x'|<r,\\
& g_k(x',\psi_k(x')) \rightarrow g_0(x',\psi_0(x')) \text{ uniformly for }|x'|<r,\\
& u_k(x',x_d-\psi_k(x'))  \rightharpoonup u_0(x',x_d-\psi_0(x')) \text{ weakly in } H^1(Q;\mathbb{R}^d),\\
& p_k(x',x_d-\psi_k(x'))  \rightharpoonup p_0(x',x_d-\psi_0(x')) \text{ weakly in }L^2(Q),
\endaligned
\right.
\end{equation}
where $Q=\big\{ (x^\prime, x_d): \, |x^\prime|<r \text{ and } 0<x_d< 10(M+1)r\big\}$.
Moreover, $(u_0, p_0)$ is a weak solution of 
\begin{equation}\label{temp3.5}
\left \{
\aligned
 -\text{\rm div} \big(A^0\nabla u_0\big) +\nabla p_0&= 0&\quad& \text{ in } D(r,\psi_0), \\
 \text{\rm div} (u_0) &= g_0&\quad & \text{ in } D(r,\psi_0),\\
 u_0&= h_0 &\quad & \text{ on } \Delta(r,\psi_0).
\endaligned
\right.
\end{equation}
\end{lemma}

\begin{proof}
We first note that (\ref{temp3.4}) follows from (\ref{module}) and (\ref{temp3.3}) by passing to subsequences. 
To prove (\ref{temp3.5}), let $\Omega\subset\overline{\Omega}\subset D(r, \psi_0)$.
Observe that if $k$ is sufficiently large, $\Omega\subset D(r, \psi_k)$.
 We now apply Theorem \ref{compactness-theorem} in $\Omega$ to conclude that
 $A^k(x/\varep_k)\nabla u_k \rightharpoonup A^0\nabla u_0$ weakly in $L^2(\Omega)$.
 As a consequence, $(u_0, p_0)$ is a weak solution of
 $-\text{\rm div} \big(A^0\nabla u_0\big) +\nabla p_0=0$ and
 $\text{\rm div} (u_0)=g_0$ in $\Omega$ for any domain $\Omega$ such that $\overline{\Omega}\subset D(r, \psi_0)$, and
 thus for $\Omega=D(r, \psi_0)$.
 Finally, let $v_k(x',x_d)=u_k(x',x_d+\psi_k(x'))$ and $ v_0(x',x_d)=u_0(x',x_d+\psi_0(x'))$. 
 That $u_0=h_0$ on $\Delta(r,\psi_0)$ in the sense of trace follows from the fact that 
 $v_k \rightharpoonup v_0$ weakly in $H^1(Q;\mathbb{R}^d)$, $v_k (x^\prime, 0)=h_k(x',\psi_k(x'))$ and
 $h_k(x^\prime, \psi_k (x^\prime)) \to h_0(x^\prime, \psi_0(x^\prime))$ uniformly on $\{ |x^\prime|<r\}$.
\end{proof}

With the help of Lemma \ref{boundary-compactness-theorem}
we prove Theorem \ref{theorem-6.1} by a compactness argument
in the same manner as in \cite{AL-1987}.

\begin{lemma}\label{step-1-b}
Let $0<\rho, \eta<1$. Then there exist constants $\varepsilon_0\in (0,1/2)$ 
and $\theta\in(0,1/4)$, depending only on $d$, $\mu$, $\rho$, $\eta$,
 and $(M, \omega)$ in (\ref{module}), such that
\begin{equation}\label{6.4-0}
\left(\average_{D_\theta} |u_\varep |^2\right)^{1/2} \le \theta^{\rho}
\end{equation}
for any $0<\varepsilon<\varepsilon_0$, whenever $(u_\varepsilon, p_\varep)
 \in H^1(D_1;\mathbb{R}^d)\times L^2(D_1)$ is a weak solution of
\begin{equation}\label{6.4-1}
\left\{
\aligned
 \mathcal{L}_\varepsilon(u_\varepsilon)  +\nabla p_\varepsilon &=0&\quad &\text{ in } D_1,\\
 \text{\rm div}(u_\varepsilon)  &=g &\quad & \text{ in } D_1,\\
 u_\varepsilon& =h &\quad & \text{ on } \Delta_1,
\endaligned
\right.
\end{equation}
and
\begin{equation}\label{6.4-2}
\left\{
\aligned
& h(0)=0, \ \ \|h\|_{C^{0,1}(\Delta_1)} \le 1,\\
& \average_{D_1} |u_\varepsilon|^2\le 1, \ \ \ \|g\|_{C^{\eta}(D_1)} \le 1.
\endaligned
\right.
\end{equation}
\end{lemma}

\begin{proof}
We will prove the lemma by contradiction. Let $\sigma=(1+\rho)/2>\rho$.
Using the boundary H\"older estimates for solutions of Stokes systems
 with constant coefficients, we obtain 
\begin{equation}\label{6.4-3}
\left( \average_{D_r}  |w|^2 \right)^{1/2}
\le C\, r^\sigma \|w\|_{C^{\sigma}(D_{1/4})}
 \le C_0 \, r^\sigma,
 \end{equation}
if $ 0<r<(1/4)$ and $(w, p_0)$ satisfies 
\begin{equation}\label{6.4-4}
\left\{
\aligned
& -\text{\rm div}\big(A^0 \nabla w\big) + \nabla p_0=0 \ \text{ in }D_{1/2},\\
& \text{div } w=g \  \text{   in }D_{1/2},\\
& w= h \  \text{ on }\Delta_{1/2},\\
& \|h\|_{C^{0,1}(\Delta_{1/2})} \le 1, \ \ h(0)=0,\\
& \int_{D_{1/2}} |w|^2 \le |D_1|, \text{ and } \ \|g\|_{C^{\eta}(D_{1/2})} \le 1,
\endaligned
\right.
\end{equation}
where $A^0$ is a constant matrix satisfying the ellipticity condition (\ref{ellipticity-1}).
The constant $C_0$ in (\ref{6.4-3}) depends only on $d$,
$\mu$, $\rho$, $\eta$, and $(M, \omega)$ in (\ref{module}).
We now choose $\theta\in (0,1/4)$ so small that
\begin{equation}\label{6.4-5}
2 C_0\theta^{\sigma} < \theta^{\rho}.
\end{equation}
We claim that the lemma holds for this $\theta$ and some $\varepsilon_0>0$, 
which depends only on $d$, $\mu$, $\rho$, $\eta$,  and $(M,\omega)$.

Suppose this is not the case. Then there exist sequences $\{\varepsilon_k\}$,
$\{A^k\}$, $\{ u_k\}$, $\{p_k\}$, $\{g_k\}$, $\{ h_k\}$, $\{ \psi_k\}$,
 such that  $\varepsilon_k \to 0$, $A^k$ satisfies (\ref{ellipticity}) and (\ref{periodicity}),
 $\psi_k$ satisfies (\ref{module}),
\begin{equation}\label{6.4-6}
\left\{
\aligned
& -\text{\rm div} \big(A^k (x/\varep_k)\nabla u_k\big) + \nabla p_k=0  \ \text{ in }D(k),\\
& \text{div} (u_k) = g_k \ \text{ in }D(k),\\
& u_k =h_k \ \text{ on }\Delta(k),\\
& \|h_k\|_{C^{0,1}(\Delta(k))} \le 1, \ \ h_k(0)=0,\\
&\left(\average_{D(k)} |u_k|^2\right)^{1/2}\le 1, \ \ \|g_k\|_{C^{\eta}(D(k))} \le 1,
\endaligned
\right.
\end{equation}
and
\begin{equation}\label{6.4-7}
\left(\average_{D(\theta, \psi_k )}|u_k|^2\right)^{1/2} > \theta^{\rho},
\end{equation}
where $D(k)=D(1, \psi_k)$ and $\Delta(k)=\Delta (1, \psi_k)$.
Note that by Cacciopoli's inequality (\ref{B-Cacciopoli}),
 the sequence $\{\| u_k\|_{H^1(D(1/2, \psi_k))} \}$ is  bounded. 
  In view of Lemma \ref{boundary-compactness-theorem},
  by passing to  subsequences, we may assume that
\begin{equation}\label{6.4-8}
\left\{
\aligned
&\widehat{A^k} \to A^0,\\
& \psi_k  \to \psi_0 \text{ and } \nabla \psi_k \to \nabla \psi_0  \text{ uniformly in } \{ |x'|<1\},\\
& u_k(x',x_d-\psi_k(x')) \rightarrow u_0(x',x_d-\psi_0(x')) \text{ weakly in } H^1(Q;\mathbb{R}^d),\\
& h_k(x^\prime, \psi_k(x^\prime)) \to h_0 (x^\prime, \psi_0(x^\prime))
\text{ uniformly in } \{ |x^\prime|<1\},\\
& g_k (x^\prime, x_d-\psi_k(x^\prime)) \to g_0 (x^\prime, x_d-\psi_0 (x^\prime))
\text{ uniformly in } Q,
\endaligned
\right.
\end{equation}
where $Q=\{ (x^\prime, x_d): \, |x^\prime|<1/2 \text{ and } 0<x_d< 5(M+1) \}$.
Moreover, we note that
$u_0 \in H^1(D(1/2,\psi_0);\mathbb{R}^d)$ and satisfies
\begin{equation*}
\left\{
\aligned
 -\text{\rm div}\big(A^0\nabla u_0\big) +\nabla p_0& =0&\quad & \text{ in }D(1/2,\psi_0),\\
 \text{\rm div} (u_0) & =g_0 &\quad &\text{ in }D(1/2,\psi_0),\\
 u_0 &= h_0 &\quad  & \text{ on }\Delta(1/2,\psi_0),\\
\endaligned
\right.
\end{equation*}
Observe that by (\ref{6.4-6}) and (\ref{6.4-8}),
$$
h_0(0)=0, \ \ \| h_0\|_{C^{0,1}(\Delta(1/2, \psi_0))}\le 1, \ \ 
\| g_0\|_{C^\eta(D(1/2, \psi_0))} \le 1,
$$
$$
\int_{D(1/2, \psi_0)} |u_0|^2
=\lim_{k\to \infty} \int_{D(1/2, \psi_k)} |u_k|^2
\le \lim_{k\to \infty} |D(1, \psi_k)|
=|D(1, \psi_0)|.
$$
It follows that $w=u_0$ satisfies (\ref{6.4-4}). However,
by (\ref{6.4-7}),
\begin{equation}\label{6.4-9}
\left(\average_{D(\theta,\psi_0)} |u_0|^2\right)^{1/2}
 = \lim_{k \rightarrow \infty} \left(\average_{D(\theta,\psi_k)} |u_k|^2 \right)^{1/2}
 \ge \theta^{\rho}.
\end{equation}
Thus, by (\ref{6.4-3}), we obtain 
 $\theta^{\rho} \le C_0 \theta^\sigma$,
which contradicts the choice of $\theta$.
This completes the proof.
\end{proof}

\begin{lemma}\label{step-2-b}
Fix $0<\rho, \eta<1$. Let $\varepsilon_0$ and $\theta$ be constants 
given by Lemma \ref{step-1-b}. 
Suppose that $(u_\varepsilon, p_\varep) \in H^1(D(1,\psi);\mathbb{R}^d)\times L^2(D(1, \psi))$ be a weak solution of
\begin{equation*}
\left\{
\aligned
 \mathcal{L}_\varepsilon (u_\varepsilon) +\nabla p_\varepsilon  &=0&\quad& \text{ in } D(1,\psi),\\
\text{\rm div }u_\varepsilon & = g&\quad&  \text{ in } D(1,\psi),\\
 u_\varepsilon&=h& \quad &  \text{ on }\Delta(1,\psi),\\
\endaligned
\right.
\end{equation*}
where $g\in C^\eta(D(1, \psi))$, $h \in C^{0,1}(\Delta(1,\psi),\mathbb{R}^d)$ and $h(0)=0$.
Then, if $0<\varepsilon<\varepsilon_0\theta^{k-1}$ for some $k \ge 1$,
\begin{equation}\label{6.5-1}
\left(\average_{D(\theta^k,\psi)}|u_\varepsilon|^2\right)^{1/2}
 \le \theta^{k\rho}\max\left\{\left(\average_{D(1,\psi)}|u_\varepsilon|^2\right)^{1/2},
 \|g\|_{C^{\eta}(D(1,\psi))}, \|h\|_{C^{0,1}(\Delta(1,\psi))}\right\}.
\end{equation}
\end{lemma}

\begin{proof}
We  prove the lemma by an induction argument on k. 
The case $k=1$ follows directly from Lemma \ref{step-1-b}.
Now suppose that the estimate (\ref{6.5-1}) is true for some $k\ge 1$. 
Let $0<\varepsilon<\varepsilon_0\theta^k$.
 We apply Lemma \ref{step-1-b} to the function
$$
w(x)=u_\varepsilon(\theta^k x) \quad \text{ in } D(1,\psi_k), 
$$
 where $ \psi_k(x')=\theta^{-k}\psi(\theta^k x')$.
Observe that $\psi_k$ satisfies (\ref{module}) uniformly in $k$, and 
\begin{equation*}
\left\{
\aligned
 \mathcal{L}_{\frac{\varepsilon}{\theta^k}}(w) +\nabla \big(  \theta^k p_\varep(\theta^k x) \big)  
 &=0&\quad & \text{ in }D(1,\psi_k),\\
 \text{div}( w) &= \theta^k g(\theta^k x)& \quad &  \text{ in }D(1,\psi_k),\\
 w &= h(\theta^k x)&\quad& \text{ on }\Delta(1,\psi_k).
\endaligned
\right.
\end{equation*}
Since $\theta^{-k}\varepsilon<\varepsilon_0$, by the induction assumption,
\begin{equation*}
\aligned
& \left(\average_{D(\theta^{k+1},\psi)} |u_\varepsilon|^2\right)^{1/2}
  = \left(\average_{D(\theta,\psi_k)} |w|^2\right)^{1/2}\\
& \le \theta^{\rho} \max \left\{
\left(\average_{D(1,\psi_k)} |w|^2\right)^{1/2},
\theta^{k}\|g(\theta^k x)\|_{C^{\eta}(D(1,\psi_k))}, 
\|h(\theta^k x) \|_{C^{0,1}(\Delta(1,\psi_k))}\right\}\\
& \le \theta^{\rho} \max \left\{\left(\average_{D(\theta^{k},\psi)} |u_\varepsilon|^2\right)^{1/2},
\theta^{k}\|g\|_{C^{\eta}(D(1,\psi))}, \theta^{k}\|h\|_{C^{0,1}(\Delta(1,\psi))}\right\}\\
& \le \theta^{(k+1)\rho} \max \left\{\left(\average_{D(1,\psi)} |u_\varepsilon|^2\right)^{1/2},
\|g\|_{C^{\eta}(D(1,\psi))}, \|h\|_{C^{0,1}(\Delta(1,\psi))}\right\}.
\endaligned
\end{equation*}
This completes the proof.
\end{proof}

We are now ready to give the proof of Theorems \ref{theorem-6.1} and \ref{main-theorem-2}.

\begin{proof}[\bf Proof of Theorem \ref{theorem-6.1}]
By considering the function $u_\varep(rx)$ in $D(1, \psi_r)$, where $\psi_r (x^\prime) =r^{-1} \psi(rx^\prime)$,
we may assume that $r=1$.
Note that $\|\nabla \psi_r\|_\infty =\|\nabla \psi\|_\infty\le M$ and
$$
|\nabla \psi_r (x^\prime)-\nabla \psi_r (y^\prime)|
=|\nabla\psi (rx^\prime)-\nabla \psi (ry^\prime)|
\le \omega(|rx^\prime -ry^\prime|)
\le \omega(r_0 | x^\prime- y^\prime|).
$$
The bounding constants $C$ will depend on $r_0$, if $r_0>1$.

Let $\varep\le t<1$.
We may assume that $t<\varep_0 \theta$, for otherwise the estimate is trivial.
Choose $k\ge 1$ so that $\varep_0 \theta^{k+1} \le t <\varep_0 \theta^k$.
Since $\varep<\varep_0 \theta^{k-1}$, it follows from Lemma \ref{step-2-b} that
$$
\aligned
\left(\average_{D_t} |u_\varep|^2\right)^{1/2}
&\le C \left(\average_{D_{\theta^k}} |u_\varep|^2 \right)^{1/2}\\
&\le C \theta^{k\rho} \left\{ \left(\average_{D_1} |u_\varep|^2\right)^{1/2}
+\| g\|_{C^\eta(D_1)} +\| h\|_{C^{0,1}(\Delta_1)}\right\}\\
&\le C\, t^\rho
\left\{ \left(\average_{D_1} |u_\varep|^2\right)^{1/2}
+\| g\|_{C^\eta(D_1)} +\| h\|_{C^{0,1}(\Delta_1)}\right\}.
\endaligned
$$
This finishes the proof.
\end{proof}

\begin{proof}[\bf Proof of Theorem \ref{main-theorem-2}]
First, we note that by Cacciopoli's inequality and Poincar\'e inequality, it suffices to show that
\begin{equation}\label{6.8-0}
\left(\average_{B(x_0, r)\cap\Omega} |u_\varep|^2\right)^{1/2}
\le C \left(\frac{r}{R}\right)^\rho \left(\average_{B(x_0, R)\cap\Omega} |u_\varep|^2\right)^{1/2}
\end{equation}
for $0<r<c_0 R<R_0$.
By translation we may assume that $x_0=0$.
Next, we may assume that in a new coordinate system, obtained from the current system
through a rotation by an orthogonal matrix with rational entries,
\begin{equation}\label{6.8-1}
\aligned
B(0, R)\cap\Omega & =B(0, R)\cap \big\{ (x^\prime, x_d): \, x_d>\psi (x^\prime) \big\},\\
B(0, R)\cap\partial \Omega & =B(0, R)\cap \big\{ (x^\prime, x_d): \, x_d=\psi (x^\prime) \big\},
\endaligned
\end{equation}
where $\psi$ is a $C^1$ function satisfying $\psi(0)=0$ and (\ref{module}).
Here we have used the fact that for any $d\times d$ orthogonal matrix $O$ and $\delta>0$,
there exists a $d\times d$ orthogonal matrix $T$ with rational entries such that
$\|O-T\|_\infty<\delta$. Moreover, each entry of $T$ has a denominator less than a constant
depending only on $d$ and $\delta$ (see \cite{Schmutz-2008}).
Finally, we point out that if $(u_\varep, p_\varep)$ is a solution of the Stokes system (\ref{Stokes})
and $u^\beta (x)=T_{\gamma \beta} v^\gamma (y)$, $p(x)=q(y)$, 
where $T= (T_{ij})$ is an orthogonal matrix and $y=Tx$,
then
\begin{equation}\label{6.8-2}
\left\{
\aligned
-\text{\rm div}_y \big (B(y/\varep)\nabla_y v\big) +\nabla_y q  &=G (y),\\
\text{\rm div}_y (v) &=h(y),
\endaligned
\right.
\end{equation}
where $B(y)= \big (b_{k\ell}^{t\gamma} (y) \big)$ with
$b_{k\ell}^{t\gamma} (y)=T_{t\alpha} T_{\gamma\beta} T_{\ell j} T_{k i} a_{ij}^{\alpha\beta} (x)$,
$G^t (y)=T_{t\alpha} F^\alpha(x) $, and $h(y) =g(x)$.
Note that the matrix $B(y)$ is periodic, if $T$ has rational entries
(a dilation may be needed to ensure that $B$ is 1-periodic).
These observations  allow us to deduce estimate (\ref{6.8-0}) from Theorem \ref{theorem-6.1}
and complete the proof.
\end{proof}


\section{$W^{1,p}$ estimates}
\setcounter{equation}{0}

In this and next sections we establish uniform $W^{1, p}$ estimates for the Stokes system (\ref{Stokes})
under the additional condition that $A$ belongs to $V\!M\!O(\mathbb{R}^d)$:
\begin{equation}\label{module-1}
\sup_{\substack{y\in \mathbb{R}^d\\ 0< t < r}}
\average_{B(y, t)} \big | A-\average_{B(y, t)} A\big | \le \omega_1 (r),
\end{equation}
where ${\omega_1}$ is a (fixed) nondecreasing continuous function on $[0, \infty)$ and $\omega_1 (0)=0$.

The following two lemmas provide the local interior and boundary $W^{1,p}$ estimates.

\begin{lemma}\label{lemma-7.1}
Suppose that $A(y)$ satisfies  the ellipticity condition (\ref{ellipticity}) and 
smoothness condition (\ref{module-1}).
Let $(u, p)\in H^1(B(0,1); \mathbb{R}^d)\times L^2(B(0,1))$
 be a weak solution to  
 \begin{equation}\label{7.1-0}
 -\text{\rm div} \big(A(x)\nabla u\big) +\nabla p=0 \quad \text{ and }
\quad \text{\rm div} (u)=0
\end{equation}
in $B(0,1)$.
Then $|\nabla u|\in L^q (B(0,1/2))$ for any $2<q<\infty$, and
\begin{equation}\label{7.1-1}
\left(\average_{B(0,1/2)} |\nabla u|^q \right)^{1/q}
\le C_q \left( \average_{B(0,1)} |\nabla u|^2 \right)^{1/2},
\end{equation}
where $C_q$ depends only on $d$, $\mu$, $q$, and $\omega_1$ in (\ref{module-1}).
\end{lemma}

\begin{lemma}\label{lemma-7.2}
Suppose that $A(y)$ satisfies  (\ref{ellipticity}) and (\ref{module-1}).
Let $(u, p)\in H^1(D_1; \mathbb{R}^d)\times L^2(D_1)$
 be a weak solution to  (\ref{7.1-0}) 
in $D_1$ and $u=0$ on $\Delta_1$.
Then $|\nabla u|\in L^q (D_{1/2})$ for any $2<q<\infty$, and
\begin{equation}\label{7.2-1}
\left(\average_{D_{1/2}} |\nabla u|^q \right)^{1/q}
\le C_q \left( \average_{D_1} |\nabla u|^2 \right)^{1/2},
\end{equation}
where $C_q$ depends only on $d$, $\mu$, $q$,  $(M, \omega)$ in (\ref{module}),
and   $\omega_1$ in (\ref{module-1}).
\end{lemma}

We remark that $W^{1,p}$ estimates for elliptic equations and systems with continuous or $V\!M\!O$
coefficients have been studied extensively in recent years.
In particular, estimates in Lemmas \ref{lemma-7.1} and \ref{lemma-7.2}
are known for solutions of $\text{\rm div} \big(A(x)\nabla u)=0$
(see \cite{ CP-1998, Byun-Wang-2004, Shen-2005, Krylov-2007,Byun-Wang-2008} and their references).
To prove Lemmas \ref{lemma-7.1} and \ref{lemma-7.2},
one follows the approach in \cite{Shen-2005} and apply a real-variable argument originated 
in \cite{CP-1998}. This reduces the problem to the case of Stokes systems with constant coefficients.
Note that for Stokes systems with constant coefficients,
the interior estimate (\ref{7.1-1}) is well known, while the boundary estimate (\ref{7.2-1})
in $C^1$ domains follows from \cite{Dindos-2004}.
We omit the details.

\begin{lemma}\label{lemma-7.3}
Suppose that $A(y)$ satisfies conditions (\ref{ellipticity}), (\ref{periodicity}) and (\ref{module-1}).
Let $(u_\varep, p_\varep)\in H^1(B(x_0,r); \mathbb{R}^d)\times L^2(B(x_0,r))$
 be a weak solution to  
 \begin{equation}\label{7.3-0}
 -\text{\rm div} \big(A(x/\varep)\nabla u_\varep\big) +\nabla p_\varep=0 \quad \text{ and }
\quad \text{\rm div} (u_\varep)=0
\end{equation}
in $B(x_0,r)$ for some $x_0\in \mathbb{R}^d$ and $r>0$.
Then  for any $2<q<\infty$,
\begin{equation}\label{7.3-1}
\left(\average_{B(x_0,r/2)} |\nabla u_\varep|^q \right)^{1/q}
\le C_q \left( \average_{B(x_0,r)} |\nabla u_\varep |^2 \right)^{1/2},
\end{equation}
where $C_q$ depends only on $d$, $\mu$, $q$, and $\omega_1$ in (\ref{module-1}).
\end{lemma}

\begin{proof}
By translation and dilation we may assume that $x_0=0$ and $r=1$.
We may also assume $\varep<(1/4)$.
The case $\varep\ge (1/4)$ follows directly from Lemma \ref{lemma-7.1},
as the coefficient matrix $A(x/\varep)$ satisfies (\ref{module-1}) uniformly in $\varep$.

Let $u(x)=\varep^{-1} u_\varep (\varep x)$ and $p(x)=p_\varep (\varep x)$.
Then $(u, p)$ satisfies (\ref{7.1-0}) in $B(0,1)$.
It follows that 
$$
\aligned
\left(\average_{B(0,\varep/2)} |\nabla u_\varep|^q \right)^{1/q}
&\le C \left( \average_{B(0,\varep)} |\nabla u_\varep|^2 \right)^{1/2}\\
&\le C \left( \average_{B(0,1/2)} |\nabla u_\varep|^2 \right)^{1/2},
\endaligned
$$
where we have used Theorem \ref{main-theorem-1} for the second inequality.
By translation the same argument also gives
\begin{equation}\label{7.3-3}
\left(\average_{B(y,\varep/2)} |\nabla u_\varep|^q \right)^{1/q}
\le C \left( \average_{B(y,1/2)} |\nabla u_\varep|^2 \right)^{1/2}
\end{equation}
for any $y\in B(0,1/2)$.
Estimate (\ref{7.3-1}) now follows from (\ref{7.3-3}) by covering $B(0,1/2)$
with  balls $\{ B(y_k, \varep/2)\}$, where $y_k\in B(0,1/2)$.
\end{proof}

The next theorem, whose proof may be found in \cite{Shen-2007},
provides a real-variable argument we will need for the $W^{1,p}$ estimates.

\begin{theorem}\label{real-variable}
Let $B_0$ be a ball in $\mathbb{R}^d$ and $F \in L^2(4B_0)$. Let $q>2$ and $f\in L^p(4B_0)$
 for some $2<p<q$. Suppose that for each ball $B\subset 2B_0$ with $|B|\le c_1|B_0|$, 
 there exist two measurable functions $F_B$ and $R_B$ on $2B$, such that $|F| \le |F_B|+|R_B|$ on $2B$,
\begin{equation}\label{temp2.25}
\aligned
\left(\average_{2B} |R_B|^q\right)^{1/q} & \le C_1\left\{\left(\average_{c_2B} |F|^2\right)^{1/2}
+\sup_{4B_0\supset B'\supset B}\left(\average_{B'} |f|^2\right)^{1/2}\right\},\\
\left(\average_{2B} |F_B|^2\right)^{1/2} &\le C_2 \sup_{4B_0\supset B'\supset B}\left(\average_{B'} |f|^2\right)^{1/2},
\endaligned
\end{equation}
where $C_1,C_2 >0$, $0<c_1<1$, and $c_2>2$. Then $F\in L^p(B_0)$ and
\begin{equation}\label{temp2.26}
\left(\average_{B_0} |F|^p\right)^{1/p} \le C\left\{\left(\average_{4B_0} |F|^2\right)^{1/2}
+\left(\average_{4B_0} |f|^p\right)^{1/p}\right\},
\end{equation}
where $C$ depends only on $d$, $C_1$, $C_2$, $c_1$, $c_2$, $p$ and $q$.
\end{theorem}

We are now ready to  prove the interior $W^{1,p}$ estimates
for the Stokes system (\ref{Stokes}).

\begin{theorem}\label{theorem-7.1}
Suppose that $A(y)$ satisfies conditions (\ref{ellipticity}), (\ref{periodicity}) and (\ref{module-1}).
Let $(u_\varep, p_\varep)\in H^1(B(x_0,r); \mathbb{R}^d)\times L^2(B(x_0,r))$
 be a weak solution to  
 \begin{equation}\label{7.6-0}
 -\text{\rm div} \big(A(x/\varep)\nabla u_\varep\big) +\nabla p_\varep=\text{\rm div} (f) \quad \text{ and }
\quad \text{\rm div} (u_\varep)=g
\end{equation}
in $B(x_0,r)$ for some $x_0\in \mathbb{R}^d$ and $r>0$.
Then  for any $2<q<\infty$,
\begin{equation}\label{7.6-1}
\aligned
&\left(\average_{B(x_0,r/2)} |\nabla u_\varep|^q \right)^{1/q}
+\left(\average_{B(x_0, r/2)} |p_\varep-\average_{B(x_0, r/2)} p_\varep|^q\right)^{1/q}\\
&\le C_q \left\{
\left( \average_{B(x_0,r)} |\nabla u_\varep |^2 \right)^{1/2}
+\left(\average_{B(x_0, r)} |f|^q\right)^{1/q} 
+\left(\average_{B(x_0, r)} |g|^q\right)^{1/q}
\right\},
\endaligned
\end{equation}
where $C_q$ depends only on $d$, $\mu$, $q$, and $\omega_1$ in (\ref{module-1}).
\end{theorem}

\begin{proof}
By translation and dilation we may assume that $x_0=0$ and $r=1$.
Note that the estimate for $p_\varep$ in (\ref{7.6-1}) follows easily from the estimate for $\nabla u_\varep$.
Also we may assume that $g=0$ by considering $u_\varep - \nabla w$, where
$w$ is a scalar function such that $\Delta w=g$ in $B(0,1)$ and
$w=0$ on $\partial B(0,1)$.

To apply Theorem \ref{real-variable}, 
for each $B=B(y, t)\subset B(0, 3/4)$ with $0<t<(1/64)$, we write
$u_\varep=v_\varep +z_\varep$, where $v_\varep\in H^1_0(4B; \mathbb{R}^d)$ and
$$
\left\{
\aligned
\mathcal{L}_\varep (v_\varep) +\nabla \pi_\varep  & =\text{\rm div}(f)& \quad & \text{ in } 4B,\\
\text{\rm div} (v_\varep) &=0&\quad &\text{ in  } 4B.
\endaligned
\right.
$$
Note that
\begin{equation}\label{7.6-2}
\average_{4B} |\nabla v_\varep|^2 \le C \average_{4B} |f|^2.
\end{equation}
Also, since $\mathcal{L}_\varep  (z_\varep)+\nabla (p_\varep-\pi_\varep)=0$
and div$(z_\varep)=0$ in $4B$,
we may apply Lemma \ref{lemma-7.3} to obtain
\begin{equation}\label{7.6-4}
\aligned
\left(\average_{2B} |\nabla z_\varep|^{\bar{q}}\right)^{1/\bar{q}}
&\le C \left(\average_{4B} |\nabla z_\varep|^2 \right)^{1/2}\\
&\le C \left(\average_{4B} |\nabla u_\varep|^2\right)^{1/2}
+C \left(\average_{4B} |f|^2\right)^{1/2},
\endaligned
\end{equation}
where $\bar{q}=q +1$ and we have used (\ref{7.6-2}) for the last inequality.

Finally, let $F=|\nabla u_\varep|$, $F_B=|\nabla v_\varep|$ and $R_B=|\nabla z_\varep|$.
Note that $|F|\le |F_B|+|R_B|$ on $4B$, and in view of (\ref{7.6-2}) and (\ref{7.6-4}),
 we have proved that
$$
\aligned
\left(\average_{2B} |R_B|^{\bar{q}}\right)^{1/\bar{q}}
 & \le C \left(\average_{4B} |F|^2\right)^{1/2}
+C \left(\average_{4B} |f|^2\right)^{1/2},\\
\left(\average_{2B} |F_B|^2\right)^{1/2}
& \le C \left(\average_{4B} |f|^2\right)^{1/2}.
\endaligned
$$
This allows us to use Theorem \ref{real-variable} to conclude that
$$
\left(\average_{B(x_0,1/16)} |\nabla u_\varep|^q\right)^{1/q}
\le C \left\{ \left(\average_{B(0, 1)}|\nabla u_\varep|^2\right)^{1/2}
+\left(\average_{B(0,1)} |f|^q\right)^{1/q} \right\}
$$
for any $x_0\in B(0,1/2)$,
which gives the desired estimate for $\nabla u_\varep$
by a simple covering argument.
\end{proof}



\section{Proof of Theorem \ref{main-theorem-3}}
\setcounter{equation}{0}

In this section we establish uniform boundary
$W^{1,p}$ estimates and gives the proof of Theorem \ref{main-theorem-3}.
Throughout this section we will assume that
$A$ satisfies conditions (\ref{ellipticity}), (\ref{periodicity}) and (\ref{module-1})
and  that $\Omega$ is a bounded $C^1$ domain.

We begin with a boundary H\"older estimate.

\begin{lemma}\label{lemma-8.0}
Let $x_0\in \partial\Omega$ and $0<R<R_0$, where $R_0=\text{\rm diam}(\Omega)$.
Let $(u_\varep, p_\varep)\in W^{1,2}(B(x_0,R)\cap \Omega; \mathbb{R}^d)\times L^2(B(x_0,R)\cap\Omega)$
 be a weak solution to  
 \begin{equation}\label{8.0-0}
 -\text{\rm div} \big(A(x/\varep)\nabla u_\varep\big) +\nabla p_\varep=0 \quad \text{ and }
\quad \text{\rm div} (u_\varep)=0
\end{equation}
in $B(x_0,R)\cap\Omega$ and $u_\varep =0$ on $B(x_0, R)\cap \partial\Omega$.
Then
\begin{equation}\label{8.0-1}
|u_\varep (x) -u_\varep (y)|
\le C \left(\frac{|x-y|}{R} \right)^\rho
\left(\average_{B(x_0, R)\cap \Omega} |u_\varep|^2\right)^{1/2}
\end{equation}
for any $x,y\in B(x_0, R/2)\cap\Omega$,
where $0<\rho<1$ and $C$ depends only on $d$, $\rho$, $A$ and $\Omega$.
\end{lemma}

\begin{proof}
By translation and dilation we may assume that $x_0=0$ and $R=1$.
The case $\varep\ge (1/4)$ follows directly from the local boundary $W^{1,p}$ estimates in Lemma \ref{lemma-7.2}
by Sobolev imbedding.
To treat the case $0<\varep<(1/4)$,
we note that if $0<r<\varep$, we may deduce from Lemma \ref{lemma-7.2} by rescaling that
\begin{equation}\label{8.0-3}
\aligned
\left(\average_{B(0, r)\cap\Omega} |\nabla u_\varep|^2\right)^{1/2}
&\le C_q \left( \frac{\varep}{r} \right)^{\frac{d}{q}} \left(\average_{B(0, \varep)\cap\Omega} |\nabla u_\varep|^q\right)^{1/q}\\
&\le C_q\left( \frac{\varep}{r} \right)^{\frac{d}{q}} \left(\average_{B(0, 2\varep)\cap\Omega} |\nabla u_\varep|^2\right)^{1/2}
\endaligned
\end{equation}
for any $2<q<\infty$, where we have used H\"older's inequality for the first inequality.
This, together with the estimate in Theorem \ref{main-theorem-2}, implies that
\begin{equation}\label{8.0-4}
\left(\average_{B(0, r)\cap\Omega} |\nabla u_\varep|^2\right)^{1/2}
\le C_\rho \, r^{\rho-1}
\left(\average_{B(0, 1)\cap\Omega} |\nabla u_\varep|^2\right)^{1/2}
\end{equation}
for any $0<r<(1/2)$, where $0<\rho<1$.
A similar argument gives
\begin{equation}\label{8.0-5}
\left(\average_{B(y, r)\cap\Omega} |\nabla u_\varep|^2\right)^{1/2}
\le C_\rho \, r^{\rho-1}
\left(\average_{B(0, 1)\cap\Omega} |\nabla u_\varep|^2\right)^{1/2}
\end{equation}
for any $y\in B(0,1/2)$ and $0<r<(1/2)$.
The estimate (\ref{8.0-1}) now follows.
\end{proof}

\begin{lemma}\label{lemma-8.2}
Let $x_0\in \partial\Omega$ and $0<R<R_0$, where $R_0=\text{\rm diam}(\Omega)$.
Let $(u_\varep, p_\varep)\in W^{1,2}(B(x_0,R)\cap \Omega; \mathbb{R}^d)\times L^2(B(x_0,R)\cap\Omega)$
 be a weak solution to  (\ref{8.0-0})
in $B(x_0,R)\cap\Omega$ and $u_\varep =0$ on $B(x_0, R)\cap \partial\Omega$.
Then  for any $2<q<\infty$,
\begin{equation}\label{8.2-1}
\left(\average_{B(x_0,R/2)\cap\Omega} |\nabla u_\varep|^q \right)^{1/q}\\
\le C_q \left( \average_{B(x_0,R)\cap \Omega} |\nabla u_\varep |^2 \right)^{1/2},
\end{equation}
where $C_q$ depends only on $d$, $q$, $A$, and $\Omega$.
\end{lemma}

\begin{proof}
By translation and dilation we may assume that $x_0=0$ and $R=1$.
Let $\delta(x)=\text{dist}(x, \partial\Omega)$.
It follows from the interior $W^{1,p}$ estimates in Lemma \ref{lemma-7.3} that
\begin{equation}\label{8.2-3}
\average_{B(y, c\, \delta(y))} |\nabla u_\varep(x)|^q\, dx
\le C \average_{B(y, 2c\, \delta(y))} \left| \frac{u_\varep (x)}{\delta(x)} \right|^q\, dx
\end{equation}
for any $y\in B(0, 1/2)\cap\Omega$, where $c=c(\Omega)>0$ is sufficiently small.
Integrating both sides of (\ref{8.2-3}) in $y$ over $B(0, 1/2)\cap\Omega$
 yields
\begin{equation}\label{8.2-4}
\int_{B(0, 1/2)\cap\Omega}
|\nabla u_\varep (x)|^q \, dx\le C \int_{B(0,3/4)\cap\Omega}  \left| \frac{u_\varep(x)}{\delta(x)} \right|^q\, dx.
\end{equation}
Finally, note that by Lemma \ref{lemma-8.0},
\begin{equation}\label{8.2-5}
|u_\varep (x)|\le C\, \big[\delta (x) \big]^\rho \left(\average_{B(0,1)\cap\Omega} |u_\varep|^2\right)^{1/2}
\end{equation}
for any $x\in B(0, 3/4)\cap\Omega$.
Choosing $\rho\in (0,1)$ so that $(1-\rho)q<1$,
we obtain estimate (\ref{8.2-1}) by substituting (\ref{8.2-5}) into the right hand side of
(\ref{8.2-4}).
\end{proof}

The following theorem gives the boundary $W^{1,p}$ estimates for the Stokes system (\ref{Stokes}).

\begin{theorem}\label{theorem-8.3}
Suppose that $A(y)$ satisfies conditions (\ref{ellipticity}), (\ref{periodicity}), and (\ref{module-1}).
Let $\Omega$ be a bounded $C^1$ domain in $\mathbb{R}^d$.
Let $(u_\varep, p_\varep)\in H^1(B(x_0,R)\cap\Omega; \mathbb{R}^d)\times L^2(B(x_0,R)\cap \Omega)$
 be a weak solution to  
 \begin{equation}\label{8.3-0}
 -\text{\rm div} \big(A(x/\varep)\nabla u_\varep\big) +\nabla p_\varep=\text{\rm div} (f) \quad \text{ and }
\quad \text{\rm div} (u_\varep)=g
\end{equation}
in $B(x_0,R)\cap \Omega$ for some $x_0\in \partial\Omega$ and $0<R<R_0$,
where $R_0=\text{diam} (\Omega)$.
Then  for any $2<q<\infty$,
\begin{equation}\label{8.3-1}
\aligned
&\left(\average_{B(x_0,R/2)\cap\Omega} |\nabla u_\varep|^q \right)^{1/q}
+\left(\average_{B(x_0, R/2)\cap\Omega} |p_\varep-\average_{B(x_0, R/2)\cap\Omega} p_\varep|^q\right)^{1/q}\\
&\le C_q \left\{
\left( \average_{B(x_0,R)\cap \Omega} |\nabla u_\varep |^2 \right)^{1/2}
+\left(\average_{B(x_0, R)\cap\Omega} |f|^q\right)^{1/q} 
+\left(\average_{B(x_0, R)\cap\Omega} |g|^q\right)^{1/q}
\right\},
\endaligned
\end{equation}
where $C_q$ depends only on $d$, $\mu$, $q$, $\omega_1$ in (\ref{module-1}), and $\Omega$.
\end{theorem}

\begin{proof}
This theorem follows from Lemmas \ref{lemma-7.3} and  \ref{lemma-8.2}
by a real-variable argument in the same manner as in the proof of Theorem \ref{theorem-7.1}.
We omit the details and refer the reader to \cite{Shen-2005}. 
\end{proof}

Finally, we give the proof of Theorem \ref{main-theorem-3}

\begin{proof}[\bf Proof of Theorem \ref{main-theorem-3}]
Since $h\in B^{1-\frac{1}{q}, q} (\partial\Omega; \mathbb{R}^d)$ and $\Omega$ is a bounded $C^1$
domain,
there exists $H\in W^{1, q}(\Omega; \mathbb{R}^d)$ such that
$$
\| H\|_{W^{1, q} (\Omega)} \le C\, \| h\|_{B^{1-\frac{1}{q}, q} (\partial\Omega)}.
$$
Thus, by considering $u_\varep -H$, we may assume that $h=0$.
Note that if $u_\varep, v_\varep \in W^{1,2}_0(\Omega; \mathbb{R}^d)$ satisfy 
\begin{equation}\label{8.4-0}
\left\{
\aligned
\mathcal{L}_\varep (u_\varep) +\nabla p_\varep  &=\text{div}(f) \\
\text{\rm div} (u_\varep)  &= g
\endaligned
\right.
\quad \text{ and } \quad 
\left\{
\aligned
\mathcal{L}^*_\varep (v_\varep) +\nabla \pi_\varep  &=\text{div}(F) \\
\text{\rm div} (v_\varep)  &= G 
\endaligned
\right.
\end{equation}
in $\Omega$, then
\begin{equation}\label{8.4-1}
\int_\Omega \nabla u_\varep \cdot F +\int_\Omega \left(p_\varep-\average_\Omega p_\varep\right)\cdot G 
=\int_\Omega \nabla v_\varep \cdot f +\int_\Omega \left(\pi_\varep -\average_\Omega \pi_\varep\right)\cdot g.
\end{equation}
This allows us to use a duality argument that reduces the theorem to 
the estimate
\begin{equation}\label{8.4-2} 
\|\nabla u_\varep\|_{L^q(\Omega)} 
+\| p_\varep -\average_\Omega p_\varep\|_{L^p(\Omega)}
\le C\Big\{ \| f\|_{L^q(\Omega)} +\| g\|_{L^q(\Omega)} \Big\}
\end{equation}
for $2<q<\infty$, where $\mathcal{L}_\varep (u_\varep) +\nabla p_\varep =\text{\rm div}(f)$,
div$(u_\varep)=g$ in $\Omega$, and $u_\varep =0$ on $\partial\Omega$.

Finally, by covering $\Omega$ with balls of radius $r_0=c_0 \text{\rm diam}(\Omega)$,
we may deduce from Theorems \ref{theorem-7.1} and \ref{theorem-8.3} that
$$
\aligned
\| \nabla u_\varep\|_{L^q(\Omega)}
&\le C \Big\{ \|\nabla u_\varep\|_{L^2(\Omega)}
+\| f\|_{L^q(\Omega)} +\| g\|_{L^q(\Omega)} \Big\}\\
&\le C\Big\{ \| f\|_{L^q(\Omega)} +\| g\|_{L^q(\Omega)} \Big\},
\endaligned
$$
where we have used the estimate in Theorem \ref{theorem-2.1} as well as $q>2$.
Also, note that
$$
\aligned
\| p_\varep -\average_{\Omega} p_\varep\|_{L^q(\Omega)}
&\le C\, \| \nabla p_\varep\|_{W^{-1, q} (\Omega)}\\
&\le C\Big\{ \|\nabla u_\varep\|_{L^q(\Omega)} +\| f\|_{L^q(\Omega)} \Big\}\\
&\le C \Big\{ \| f\|_{L^q(\Omega)} +\| g\|_{L^q(\Omega)} \Big\},
\endaligned
$$
where we have used $\nabla p_\varep =\mathcal{L}_\varep (u_\varep)-\text{\rm div} (f)$ in $\Omega$
for the second inequality.
This completes the proof.
\end{proof}

\bibliographystyle{plain}
\bibliography{Stokes.bbl}

\medskip

\begin{flushleft}
Shu Gu,
Department of Mathematics,
University of Kentucky,
Lexington, Kentucky 40506,
USA.

E-mail: gushu0329@uky.edu
\end{flushleft}

\begin{flushleft}
Zhongwei Shen,
Department of Mathematics,
University of Kentucky,
Lexington, Kentucky 40506,
USA.

E-mail: zshen2@uky.edu
\end{flushleft}

\medskip

\noindent \today

\end{document}